\documentclass[webpdf,modern,medium]{oup-authoring-template}

\onecolumn

\usepackage[all,pdf]{xy}
\usepackage{etoolbox}

\makeatletter
\patchcmd{\@maketitle}{\raisebox{5pt}[0pt]}{\@gobble}{}{\PatchFailed}
\patchcmd{\@maketitle}{\raisebox{5pt}[0pt]}{\@gobble}{}{\PatchFailed}
\patchcmd{\@maketitle}{\raisebox{5pt}[0pt]}{\@gobble}{}{\PatchFailed}
\patchcmd{\@maketitle}{{\color{black!20}\rule{45pt}{55pt}}}{}{}{\PatchFailed}
\patchcmd{\@maketitle}{{\color{black!20}\rule{45pt}{55pt}}}{}{}{\PatchFailed}
\patchcmd{\@maketitle}{{\color{black!20}\rule{45pt}{55pt}}}{}{}{\PatchFailed}
\def\ps@opening{\let\@oddhead\@empty\let\@evenhead\@empty
  \def\@oddfoot{\hfil\thepage\hfil}\let\@evenfoot\@oddfoot}
\makeatother

\numberwithin{equation}{subsection}

\makeatletter
\g@addto@macro\normalsize{%
  \abovedisplayskip 8\p@ \@plus2\p@ \@minus4\p@
  \belowdisplayskip \abovedisplayskip
  \abovedisplayshortskip 3\p@ \@plus2\p@
  \belowdisplayshortskip 5\p@ \@plus2\p@ \@minus2\p@}
\def\section{\@startsection{section}{1}{\z@}%
  {-13\p@ plus -3\p@}{4\p@}{\reset@font\raggedright\secsize}}
\def\subsection{\@startsection{subsection}{2}{\z@}%
  {-9\p@ plus -2\p@}{2\p@}{\reset@font\raggedright\subsecsize}}
\def\subsubsection{\@startsection{subsubsection}{3}{\z@}%
  {-9\p@ plus -1\p@}{0.1em}{\reset@font\raggedright\subsubsecsize}}
\AtBeginDocument{%
\newtheoremstyle{thmstyleone}{8pt plus2pt minus1pt}{8pt plus2pt minus1pt}%
  {\normalfont\itshape}{0pt}{\bfseries}{}{.5em}%
  {\thmname{#1}\thmnumber{\@ifnotempty{#1}{ }\@upn{#2}}\thmnote{ {\the\thm@notefont(#3)}}}%
\newtheoremstyle{thmstylethree}{8pt plus2pt minus1pt}{8pt plus2pt minus1pt}%
  {\normalfont}{0pt}{\bfseries}{}{.5em}%
  {\thmname{#1}\thmnumber{\@ifnotempty{#1}{ }\@upn{#2}}\thmnote{ {\the\thm@notefont(#3)}}}%
}
\makeatother

\theoremstyle{thmstyleone}
\newtheorem{thm}{Theorem}[subsection]
\newtheorem{prop}[thm]{Proposition}
\newtheorem{lem}[thm]{Lemma}
\newtheorem{cor}[thm]{Corollary}

\theoremstyle{thmstylethree}
\newtheorem{exa}[thm]{Example}
\newtheorem{rmk}[thm]{Remark}

\setcitestyle{numbers,square,comma}

\renewcommand\keywordsname{2020 Mathematics Subject Classification:}

\newcommand{\End}{{\operatorname{End}}}
\newcommand{\Ext}{{\rm Ext}}
\newcommand{\Fr}{{\operatorname{Fr }}}
\newcommand{\Hom}{{\operatorname{Hom}}}
\newcommand{\Ker}{{\operatorname{Ker}}}
\newcommand{\Lie}{{\operatorname{Lie}}}
\newcommand{\ord}{{\operatorname{ord }}}
\newcommand{\Spec}{{\operatorname{Spec }}}
\newcommand{\Sym}{{\operatorname{Sym }}}
\newcommand{\val}{{\operatorname{val}}}

\newcommand{\uy}{{\underline{y}}}

\newcommand{\calC}{{\mathcal{C}}}
\newcommand{\calF}{{\mathcal{F}}}
\newcommand{\calL}{{\mathcal{L}}}
\newcommand{\calN}{{\mathcal{N}}}
\newcommand{\calO}{{\mathcal{O}}}
\newcommand{\calT}{{\mathcal{T}}}
\newcommand{\calV}{{\mathcal{V}}}
\newcommand{\calW}{{\mathcal{W}}}
\newcommand{\calZ}{{\mathcal{Z}}}

\def\AA{\mathbb{A}}
\def\CC{\mathbb{C}}
\def\FF{\mathbb{F}}
\def\GG{\mathbb{G}}
\def\PP{\mathbb{P}}
\def\QQ{\mathbb{Q}}
\def\RR{\mathbb{R}}
\def\WW{\mathbb{W}}
\def\ZZ{\mathbb{Z}}

\newcommand{\scrC}{{\mathscr{C}}}
\newcommand{\scrD}{{\mathscr{D}}}
\newcommand{\scrF}{{\mathscr{F}}}
\newcommand{\scrG}{{\mathscr{G}}}
\newcommand{\scrK}{{\mathscr{K}}}
\newcommand{\scrL}{{\mathscr{L}}}

\newcommand{\id}{{\noindent}}
\newcommand{\arr}{{\; \rightarrow \;}}
\newcommand{\Arr}{{\; \longrightarrow \;}}
\newcommand{\surjects}{{\; \twoheadrightarrow \;}}
\newcommand{\injects}{{\; \hookrightarrow \;}}
\newcommand{\fpbar}{\overline{\mathbb{F}}_p}

\begin{document}

\firstpage{1}

\title[Foliations on abelian surfaces in positive characteristic]{Foliations on abelian surfaces in positive characteristic}

\author[1]{Fabrizio Andreatta}
\author[2]{Eyal Z. Goren}

\authormark{F. Andreatta and E. Z. Goren}

\address[1]{\orgdiv{Department of Mathematics ``Federigo Enriques''}, \orgname{University of Milan}, \orgaddress{\street{via Saldini 50}, \postcode{20133} \state{Milano}, \country{Italy}}. \href{mailto:fabrizio.andreatta@unimi.it}{fabrizio.andreatta@unimi.it}}
\address[2]{\orgdiv{Department of Mathematics and Statistics}, \orgname{McGill University}, \orgaddress{\street{805 Sherbrooke St.\ W}, \state{Montr\'eal, QC}, \postcode{H3A 0B9}, \country{Canada}}. \href{mailto:eyal.goren@mcgill.ca}{eyal.goren@mcgill.ca}}

\abstract{This paper studies foliations on abelian surfaces in positive characteristic. We provide an explicit description of all smooth foliations, linear and turbulent, and establish criteria for being $p$-closed. We explicitly determine the moduli spaces of such foliations. We give an application to the construction of $p$-closed foliations, with mild singularities, on Kummer $K3$-surfaces in positive characteristics.}

\maketitle

\section*{Introduction}
The study of foliations is an important theme in topology and differential equations and is an active and vast area of research. Foliations on real manifolds can be thought of as a decomposition of a manifold into a disjoint union of leaves, where locally the leaves have the structure of the collection of $d$-dimensional planes parallel to $\RR^d \times \{ 0^{n-d}\}$ in $\RR^n$. From a different perspective, foliations can be thought of as a geometric incarnation of the study of differential equations, where the connection is made through Frobenius' theorem, which gives necessary and sufficient conditions for the existence and uniqueness of a maximal set of independent solutions to an overdetermined system of first-order homogeneous linear partial differential equations. 

Assuming that the manifolds are algebraic varieties $X$, a foliation is, by definition, a saturated involutive subsheaf of the tangent sheaf $\calT_X$. Over the complex numbers, questions of much interest concern when there are smooth foliations on a variety $X$ or, more broadly, what can be said about their singularities, when their leaves are algebraic, and whether there is a way to resolve their singularities (or at least bring them to some sort of normal form). In that context, one can mention the Grothendieck-Katz $p$-curvature conjecture (with an extension due to Ekedahl--Shepherd-Barron--Taylor) about algebraic solutions to systems of PDEs and works of Katz and Bost \cite{Katz, Katz2, Bost}, works of Brunella and McQuillan on resolution of singularities \cite{Bru1, McQ}, works of Bott, Jouanolou, Ghys, Cerveau, Lins Neto, Camacho, Sad, Loray, Pereira, Touzet, and many others, covering diverse subjects. We were inspired by works of Brunella \cite{Brunella}, Lo Bianco--Pereira \cite{LBP} and Mendson--Pereira \cite{MeP}.

In characteristic $p$, a new element comes into play, which is a geometric version of Jacobson's ``inseparable Galois theory'' that connects $p$-closed foliations on a smooth variety $X$ to varieties~$Y$ lying between $X$ and its Frobenius transform $X^{(p)}$; some of the contributors to that theory are Ekedahl, Miyaoka, Rudakov--Shafarevich and Shepherd-Barron \cite{Ekedahl, Miyaoka, RS, SB}. 

While much work has been done over the complex numbers by the Brazilian and French schools, especially for projective spaces and other special varieties (such as homogeneous varieties), the study of similar questions in characteristic $p$ is not as developed. Our interest in this paper is in the case of abelian surfaces. We view that as a test case for a more general study of foliations on abelian varieties, which we hope to undertake, possibly combining it with work of Goren--de Shalit \cite{GdS} on foliations on certain Shimura varieties. 
However, even the study of abelian surfaces reveals very rich, new and intriguing phenomena. 

\medskip

\id Let $A$ be an abelian variety over an algebraically closed field $k$. A foliation $\scrF$ on $A$ is a saturated involutive subsheaf $\scrF \subseteq \calT_A$ of the tangent bundle. We will focus mainly on smooth foliations, meaning those for which $\scrF$ is a locally direct summand of $\calT_A$. There is always a class of easily understood smooth foliations on $A$ that are linear foliations: they arise from extending by translation a subspace $V$ of the tangent space to $A$ at the origin. Our main interest is in non-linear foliations which, for $A$ a surface and $\scrF$ smooth, are called turbulent foliations. 

If $k$ has characteristic $p$, a foliation $\scrF$ may or may not be $p$-closed. One calls $\scrF$ $p$-closed if, whenever $D$ is a local section of $\scrF$, also $D^p$ is a local section (where $D^p$ is the $p$-fold iteration of $D$, thought of as a derivation). For example, if $\scrF$ is a linear foliation, then it is $p$-closed if and only if it comes from a subspace $V$ that is a $p$-Lie subalgebra of $\Lie(A)$. Again, for linear foliations, being $p$-closed is a well understood condition via the theory of finite commutative group schemes. 

Thus, our main focus in this paper is to classify foliations on abelian surfaces $A$ over an algebraically closed field $k$ of characteristic $p>0$ and to determine their moduli and to formulate, in terms of the moduli space, the conditions for such a foliation $\scrF$ to be $p$-closed. To a $p$-closed foliation $\scrF$ there is a corresponding quotient variety $X = A/\scrF$ and a morphism $f\colon A \arr X$ that factors the Frobenius morphism ${\rm Frob} \colon A \arr A^{(p)}$:
\[ \xymatrix{A \ar[r]_f\ar@/^1pc/[rr]^{\rm Frob} & X \ar[r] & A^{(p)}}. \]

\id {\bf Description of our main results.} \hspace{0.05cm} The article is divided into three parts. In the first part we study a single foliation on $A$, in the second we study families of foliations, and in the third we study properties of $A/\scrF$ and provide applications to $K3$ surfaces. 

\medskip 
\id \textit{\S 1. A single foliation.} \; Given a turbulent foliation $\scrF$ on an abelian surface $A$, we associate to it an elliptic curve $E \subset A$ and a quotient elliptic curve $Y = A/E$. In particular, the existence of a turbulent foliation implies that $A$ is not simple (Corollary~\ref{cor: turbulent implies non-simple}). Let $\tau\colon A \arr Y$ denote the canonical morphism. We show that $\scrF$ is determined by the data $(E, \uy, f)$, where $\uy$ is an effective  divisor of degree $d$ on $Y$ and $f\in H^0(Y, \calO_Y(\uy))$ has polar divisor exactly~$\uy$.

Vice versa, given any such triple $(E, \uy, f)$ we can construct a matching foliation in the following way. As $E$ is fixed, we may fix differentials $\tilde \omega_E, \tilde \omega_Y$ on $A$ that are naturally related to $E$ and $Y$ (the choice is not unique, but is made once and for all once $E$ is given). One lets $\eta = \tau^\ast(f)\tilde \omega_Y + \tilde \omega_E $ and defines $\scrF = \eta^\perp \cap \calT_A$. See Theorem~\ref{thm classtrbfol}.

We prove in \S~\ref{sec discreteinv} that the elliptic curve $E$ is a discrete invariant of $\scrF$, and clearly $d = \deg(\uy)$ is too. We thus find a bijection between turbulent foliations, with invariants $E$ and $d$, and points of a certain variety $\WW^\circ_{Y, d}$, fibered over $\Sym^d Y$, whose fiber over $\uy$ is the Zariski open set in $H^0(Y,\calO_Y(\uy))$ of functions with polar divisor exactly $\uy$. 

To determine which foliations, in terms of a point $(\uy, f)$ of $\WW^\circ_{Y, d}$, are $p$-closed, we recall the Cartier operator $\scrC$ and establish some of its functorial properties for which we have no adequate reference. Applying several theorems of Cartier, we are able to utilize the operator $\scrC$ to characterize $p$-closed foliations in Proposition~\ref{cor: when are foliations p closed}.

\medskip

\id {\bf Proposition.} \textit{Choose a basis $\{\tilde \omega_Y, \tilde \omega_E\}$ such that $\scrC_A$ is in normal form as in Lemma~\ref{lem:normal form}.  A turbulent foliation defined by $\eta = \tau^\ast(f\omega_Y) + \tilde \omega_E$ is $p$-closed if and only if one of the following holds:}
\begin{enumerate} 
\item \textit{Case I, $E$ ordinary and $\scrC_Y(f\omega_Y) = f\omega_Y$; equivalently, \cite[Proposition 7]{Cartier}, $f\omega_Y$ is a logarithmic differential, i.e. of the form $\frac{dg}{g}$ for a function $g \in k(Y)^\times$.}

\item \textit{Case I, $E$ supersingular and $\scrC_Y(f\omega_Y) = 0$; equivalently, \cite[Proposition 8]{Cartier}, $f\omega_Y$ is an exact differential, i.e. of the sort $dg$ for a function $g \in k(Y)$.}

\item \textit{Case $\text{II}_{ord}$, $E$ ordinary and $\scrC_Y(f\omega_Y) = f\omega_Y$.}

\item \textit{Case $\text{II}_{ssg}$ and $\scrC_Y(f\omega_Y) + \omega_Y =0$. In this case, one may write $\omega_Y = dh$ and then $f$ is a solution if and only if $f = -h^{p-1} + g_{p-2}^ph^{p-2}+ \cdots + g_0^p$ for any meromorphic functions $g_i \in k(Y)$.}
\end{enumerate}

\medskip

\id Further analysis allows us to determine when $p$-closed foliations exist as a function of the datum $(E, \delta)$. Our results are given in the following theorems (Theorem~\ref{p closed thm 1}, Theorem~\ref{p closed thm 2} and Theorem~\ref{thm: p closed tricky}). 

\medskip
\id {\bf Theorem.} \textit{Assume that $E$ is supersingular and we are in Case I (Proposition~\ref{cor: when are foliations p closed} (2)). Let $\delta$ be an effective divisor on $Y$ of degree $d\geq 2$. If the divisor $\delta = \sum_{P\in Y} n_P[P] = (dg/\omega_Y)_\infty$ for some function $g\in k(Y)$, then for every point~$P$, $n_P \not\equiv 1 \pmod{p}$. 
The converse holds when $\delta$ is not of the form $2[P]$, and also holds for $\delta = 2[P]$ if (and only if)~$Y$ is ordinary.}

\textit{Therefore, if $\delta$ is of the form $ \sum_{P\in Y} n_P[P] $ and, for some $P$, $n_P \equiv 1 \pmod{p}$, then any corresponding foliation defined by a point $(\delta, f) \in \WW^\circ_{Y, d}$ is  \emph{not} $p$-closed. If, on the other hand, $\delta$ is of the form $ \sum_{P\in Y} n_P[P] $ and, for all~$P$, $n_P \not\equiv 1 \pmod{p}$, then there exists a $p$-closed foliation defined by some point $(\delta, f) \in \WW^\circ_{Y, d}$, 
unless $\delta = 2[P]$ and~$Y$ is supersingular.}

\medskip
\id {\bf Theorem.}  \textit{Assume that $E$ is ordinary and we are in Case I or Case $\text{II}_{\text{ord}}$. Let $\delta$ be an effective divisor on~$Y$ of degree $d\geq 2$. If the divisor $\delta = \sum_{P\in Y} n_P[P] = (dg/g\omega_Y)_\infty$ for some function $g\in k(Y)$, then for every point~$P$, $n_P = 0$ or $1$.}

\textit{Therefore, if $\delta$ is an effective divisor of the form $ \sum_{P\in Y} n_P[P] $ and, for some $P$, $n_P > 1$, then \emph{any} corresponding foliation defined by a point $(\delta, f) \in \WW^\circ_{Y, d}$ is  \emph{not $p$-closed}.} 

\textit{If, on the other hand, $\delta$ is an effective divisor of the form $ \sum_{i=1}^r [P_i]$, then \emph{there is} a \emph{$p$-closed} foliation defined by a point $(\uy, f) \in \WW^\circ_{Y, d}$ such that $\uy =\delta$ if and only if there exist integers $a_1, \dots, a_r$, $(a_i, p) = 1$, $\sum_{i=1}^r a_i \equiv 0\pmod{p}$. Such integers exist if and only if $r>1$, and $r$ is even when $p=2$.}

\medskip

\id {\bf Theorem.}  \textit{Let $p>2$. Assume we are in Case $\text{II}_{\text{ssg}}$, and thus $Y$ is a supersingular elliptic curve. Let $\delta_2$ be the divisor $[Q_1]+[Q_2]+[Q_3]$, where the $Q_i$ are the nonzero $2$-torsion points of $Y$.} 

\textit{Let $\delta\geq 0$ be any effective divisor on $Y$ such that ${\rm supp}(\delta) \cap \delta_2 = \emptyset$. There is then a function $f\in k(Y)$ such that $\scrC_Y(f\omega_Y) +\omega_Y=0$ and $(f)_\infty = p((p-1)\delta_2 + \delta)$. Therefore, there are $p$-closed smooth turbulent foliations $\scrF$ on $A$ such that 
\[ \scrF \cong \calO_A(-D), \quad D = p\tau^\ast((p-1)\delta_2 + \delta).  \]}

\
 
\id \textit{\S2. Families of foliations.} \; By a general theory whose fundamentals go back to Grothendieck, the quotient sheaves $\calT_A \surjects N_\scrF$ of fixed Hilbert polynomial $P$ are parameterized by a so-called Quot scheme ${\rm Quot}^P$, and those such that the quotient sheaf $N_\scrF$ is locally free are classified by an open subscheme ${\rm Quot}^{P, \circ}$. In general, those Quot schemes are notoriously difficult to understand. In the first part of the paper, we have established a bijection between turbulent smooth foliations $\scrF$ and points of a variety $\WW^\circ_{Y, d}$. In this second part of the paper, we view $\WW^\circ_{Y, d}$ as a scheme and interpret its $R$-points (for $R$ a $k$-algebra) in terms of foliations, extending the set-theoretic bijection on $k$-points. Further, we construct a natural morphism $\rho\colon \WW^\circ_{Y, d} \arr {\rm Quot}^{P, \circ}$ and prove the following theorem (see \S \ref{W and quot}). 

\medskip

\id {\bf Theorem.} \textit{The morphism $\rho$ identifies $\WW^\circ_{Y, d}$ with an irreducible component of ${\rm Quot}^{P, \circ}$; this component is therefore smooth and irreducible of dimension $2d$, naturally fibered over $\Sym^d Y$ with geometric fibers isomorphic to the complement of finitely many codimension $1$ subspaces in $\AA^d_k$.}

\medskip 

As a result, we prove the following statement (Corollary~\ref{cor: 098245}). In light of this statement, the previously cited theorems can be viewed as characterizing the fibers of $\calZ_{Y,d} \arr \Sym^dY$.

\medskip 

\id {\bf Corollary.} \textit{Let $k$ be an algebraically closed field of characteristic $p$ and $A$ an abelian surface over $k$. There exists a closed subset $\calZ_{Y,d}$ of $\WW^\circ_{Y,d}$ parameterizing $p$-closed foliations $\scrF$ with invariants $E$ and $d$.} 

\

\id \textit{Remark.} The definition of the Hilbert polynomial we use in the paper and that appears in the theorem above is ad hoc. If one instead chooses once and for all a polarization on $A$, one likewise obtains Quot schemes ${\rm Quot}^P$ and ${\rm Quot}^{P, \circ}$ and we believe that our methods would show that every component of ${\rm Quot}^{P, \circ}$, as long as $P$ has a nonzero linear term, will be isomorphic to one of the spaces $\WW^{\circ}_{Y, d}$, and thus smooth and irreducible of dimension $2d$ (and the remaining Quot schemes are associated to linear foliations and thus there is only one remaining Quot scheme which is, in fact, isomorphic to $\PP^1$).

\

\id \textit{\S3. Applications.} \; Let $\scrF$ be a turbulent $p$-closed foliation on $A$. Associated to $\scrF$ there is a quotient variety $B = A/\scrF$ that factors Frobenius $A \arr B \arr A^{(p)}$. In Proposition~\ref{prop: geometry of the quotient of A by a foliation}, we provide some simple properties of $B$ showing that it is not an abelian variety. In particular, 
\[ \pi^\ast K_B = -(p-1) \scrF.\]
Such calculations of the canonical class are already in the literature. But in light of our results about $p$-foliations, we get here entire families of such quotients that are not abelian varieties, all lying between $A$ and $A^{(p)}$. In Proposition~\ref{Prop: albanese}, we characterize the Albanese variety of $B$:
\[ {\rm Alb}(B, 0_B) \cong A^{(p)}.\]
Finally, we use our results to construct certain $p$-closed foliations on Kummer $K3$ surfaces. $K3$ surfaces  over fields of characteristic $0$, or $p>3$, do not support smooth foliations. The foliations we construct have a unique singular point on each of the $16$ exceptional fibers of the Kummer surface, and each exceptional fiber is an integral subvariety. See Proposition~\ref{prop: foliations on kummer}.

\

\id {\bf Use of AI disclosure.} AI (ChatGPT)  was used in this paper for copy-editing tasks, such as spelling and grammar corrections, verification of hyperlinks, etc. We have also used ChatGPT to verify the mathematical correctness of the content. We challenged Claude to provide an alternative proof for the case $\delta = 2[P]$ in Theorem~\ref{p closed thm 1}. Its argument was then simplified by us and used to replace our original proof that had the drawback of using a lot of machinery. ChatGPT also helped us figure out the situation with integral curves in \S~\ref{sec integralcurves}.

\

\id {\bf Acknowledgments.} E. Z. Goren was partially supported by an NSERC Discovery Grant 223148. Both authors thank McGill University and Universit\`a Statale di Milano for their hospitality and excellent working conditions.  

\

\section{Foliations and $p$-closed foliations on an abelian surface in characteristic $p$}\label{sec:foliations generalities}

\subsection{Foliations} Let $k$ be an algebraically closed field and let $X$ be a non-singular variety over~$k$. Let $\calT_X$ denote its tangent bundle. A {\bf foliation} $\scrF$ on $X$ is a saturated involutive subsheaf of $\calT_X$, where by {\bf saturated} we mean that $N_\scrF:= \calT_X/\scrF$ is a torsion-free $\calO_X$-module, and by {\bf involutive} we mean that it is closed under the Lie bracket (coming from interpreting sections of $\calT_X$ as derivations). The involutive condition is automatic if $\scrF$ has rank $1$. The closed subset where $N_\scrF$ is not free is called the singularity of $\scrF$ and denoted ${\rm Sing}(\scrF)$. The following is well known but is included for convenience.

\begin{lem}\label{lem: reflexive and loc free} Let $X$ be a regular surface over $k$. Let $\scrF$ be a codimension $1$ foliation on $X$. Then $\scrF$ is a locally free sheaf, and ${\rm Sing}(\scrF)$ is of codimension at least $2$. 
\end{lem}
\begin{proof} Let $\scrF^{\scriptscriptstyle\vee} = \mathscr{H}\!\textit{om}(\scrF, \calO_X)$ be the dual sheaf to $\scrF$. There is a natural inclusion \[\scrF \injects \scrF^{{\scriptscriptstyle\vee}{\scriptscriptstyle\vee}} \injects \calT_X^{{\scriptscriptstyle\vee}{\scriptscriptstyle\vee}} = \calT_X,\] and the quotient sheaf $\scrF^{{\scriptscriptstyle\vee}{\scriptscriptstyle\vee}}/\scrF$ is torsion. Since $\scrF^{{\scriptscriptstyle\vee}{\scriptscriptstyle\vee}}/\scrF \injects \calT_X/\scrF$, which is torsion-free, $\scrF = \scrF^{{\scriptscriptstyle\vee}{\scriptscriptstyle\vee}}$ and, as for a coherent sheaf $\scrG$ on a normal scheme $\scrG^{\scriptscriptstyle\vee}$ is always reflexive \cite[Corollary 1.2]{HartshorneReflexive}, it follows that~$\scrF$ is a reflexive sheaf. By Corollary~$1.4$ of loc. cit., a reflexive sheaf on a regular scheme of dimension $2$ is locally free. 

The quotient sheaf $N_\scrF$ is a coherent torsion-free sheaf; the proof of Proposition 1.3 loc. cit. explains why $(N_\scrF)^{{\scriptscriptstyle\vee}{\scriptscriptstyle\vee}}/N_\scrF$ is supported in codimension $2$ (briefly, at points $x\in X$ such that $\dim \calO_{X, x} = 1$, torsion-free implies that $(N_\scrF)_x$ is free and identifies with $(N_\scrF^{{\scriptscriptstyle\vee}{\scriptscriptstyle\vee}})_x$).
\end{proof}

\id The subtlety in the rank $1$ case is thus whether $\scrF$ is a locally direct summand of $\calT_X$ or merely locally free. 

\subsection{Smooth foliations on abelian surfaces} Assume now that $A$ is an abelian surface and $\scrF$ is a smooth rank $1$ foliation on $A$, meaning that ${\rm Sing}(\scrF) = \emptyset$. We then have an exact sequence of vector bundles
\begin{equation}\label{eqn: foliation ses} \xymatrix{ 0 \ar[r]& \scrF \ar[r] & \calT_A \ar[r] & N_\scrF \ar[r]& 0,}
\end{equation}
and thus a dual sequence
\begin{equation}\label{eqn: dual to foliation ses} \xymatrix{ 0 \ar[r]& \scrF^\perp \ar[r] &\Omega^1_A \ar[r] & \scrF^\ast \ar[r]& 0,}
\end{equation}
where $\perp$ is taken with respect to the pairing between $\calT_A$ and $\Omega^1_A$ and $\ast$ denotes the dual line bundle. 

Since $A$ is an abelian variety, the sheaves $\calT_A$ and $\Omega^1_A$ are trivial vector bundles. In particular, it follows that $\calO_A \cong \det \Omega^1_A \cong \scrF^\perp \otimes \scrF^\ast$, and thus
\[ \scrF^\perp \cong N_\scrF^\ast\cong \scrF, \quad \scrF^\ast \cong N_\scrF.\]
Locally $\scrF$ is a trivial line bundle and thus over a non-empty open set $U$, $\scrF^\perp(U)$ has a nowhere vanishing section $\eta$ that we may view as a meromorphic differential form on $A$. Thus, over $U$ we have 
\[ \scrF = \eta^\perp,\]
and, conversely, every nonzero meromorphic differential $\eta$ on $A$ defines a foliation $\scrF$ on $A$; indeed, over the open set where $\eta$ is regular and nowhere vanishing $\scrF = \eta^\perp$ and it uniquely extends to a foliation on the whole of $A$. This relation establishes a bijection between foliations and meromorphic differentials considered up to multiplication by a nonzero meromorphic function. The following lemma is known to the experts and holds in much greater generality.

\begin{lem}The foliation $\scrF$ defined by a meromorphic differential $\omega$ is smooth if and only if at every point $x\in A$ there is a meromorphic function $f_x$ such that $f_x\omega$ is regular and non-vanishing on an open neighborhood of~$x$.
\end{lem}
\begin{proof} Locally, at a point $x\in A$, after rescaling $\omega$ we may write $\omega = f_1\omega_1 + f_2\omega_2$, where $\omega_i$ are regular differentials and $f_i$ are regular functions on an open set $U$ containing $x$, where the intersection of their divisors $(f_1)\cap (f_2)$ is either the empty set or supported on the point $x$. 

In the first case, after possibly shrinking $U$ further, we can write $1 = g_1 f_1+g_2f_2$ for some regular functions $g_i$ on $U$ and we see that $\omega^\perp$ is free on $U$, generated by $f_2\partial_1 - f_1\partial_2$, where $\{\partial_i\}$ form the basis for $\calT_A$ dual to $\{\omega_i\}$ and $N_\scrF = (\oplus_i\calO_U\partial_i)/\calO_U(f_2\partial_1 - f_1\partial_2)$ is free too; it is isomorphic to $\calO_U$, where the isomorphism takes $a_1\partial_1 + a_2\partial_2$ to $a_1f_1 + a_2f_2$.

In fact, this description of $\omega^\perp$ is valid in the second case too, except that $N_\scrF$ is not free at $x$. The same map embeds $N_\scrF$ as a certain subsheaf of functions of $\calO_U$ that have the property that they all vanish at $x$ to a certain order determined by $f_1, f_2$, but the quotient sheaf $\calO_U/N_\scrF$ is supported on $\{x\}$ and therefore $N_\scrF$ is not locally free (if it were, then locally all its sections would vanish on a curve through $x$). 

It only remains to note that we are in the first case precisely if the condition stated in the lemma holds. 
\end{proof}

Choose a trivialization $\calT_A \cong \calO_A^2$, arbitrarily at this stage; later we will be more precise about the choice. The incarnation of $\PP^1_k$ as a classifying space (in fact, a special kind of ${\rm Quot}$ scheme) implies that there is a {\bf classifying morphism} 
\[\Phi = \Phi_\scrF\colon A \arr \PP^1,\] 
such that 
\[ (\calT_A \arr N_\scrF) \cong \Phi^\ast( \calO_{\PP^1}^2 \arr \calO_{\PP^1}(1)).\]
Equivalently, 
\[ (\scrF\arr \calT_A ) \cong \Phi^\ast( \calO_{\PP^1}(-1) \arr \calO_{\PP^1}^2 ).\]

\

\subsection{Turbulent foliations}\label{sec turbfol} A {\bf linear} foliation $\scrF$ on an abelian variety $A$ is a foliation that is given by a trivial line bundle $\scrF \injects \calT_A$; equivalently, choosing a $1$-dimensional subspace $V\subset \calT_{A,0}$ of the tangent space at the origin, $\scrF$ is the unique subsheaf of $\calT_A$ whose fiber at a closed point $x\in A(k)$ is $T_x(V)\subset \calT_{A, x}$, where $T_x\colon A \arr A$ is the translation-by-$x$ morphism $T_x(y) = x+y$. Such foliations are well-understood--see further discussion in \S~\ref{subsec:linear p closed} below.

A rank $1$ smooth foliation that is not linear is called {\bf turbulent}. Our main interest is in turbulent foliations. The following lemma is clear. 

\begin{lem} A smooth foliation $\scrF$ on $A$ is turbulent if and only if the classifying morphism $\Phi_\scrF\colon A \arr \PP^1$ is dominant. 
\end{lem}

\id Recall the following fact \cite[Proposition 2.20]{EvdGM}, \cite[Proposition A.7.1.6]{HindrySilverman}.
\begin{prop} Let $A$ be an abelian variety over an algebraically closed field $k$. Let
$f\colon A \arr X$ be a morphism of $k$-varieties. For $a\in A$, let $C_a$ denote the connected component of
the fiber over $f(a)$ such that $a\in C_a$, and write $F_a$ for the reduced scheme underlying $C_a$. Then
$F_0$ is an abelian subvariety of $A$ and $F_a = T_a(F_0) = a+ F_0$ for all $a\in A(k)$.
\end{prop}

Let $\scrF$ be a turbulent foliation and let
\[ E = E_\scrF: = F_0 = (\Phi^{-1}\Phi(0))^\circ;\]
namely, the connected component of the unique fiber of $\Phi$ that contains $0$, with the reduced subscheme structure. Since $\scrF$ is turbulent, $\Phi_\scrF$ is dominant and it follows that $E$ is an elliptic curve. 

\begin{cor}\label{cor: turbulent implies non-simple} An abelian surface $A$ that has a turbulent foliation is not simple. 
\end{cor}

We let $Y = Y_\scrF := A/E_\scrF$ and denote by $\tau = \tau_\scrF$ the morphism 
\[ \xymatrix{A \ar[r]^\tau & Y}.\]
This map $\tau$ provides the normal reduction of $N_\scrF$ in the language of Brunella \cite{Brunella}; there is a commutative diagram
\[\xymatrix{A \ar[rr]^\Phi \ar[dr]_\tau& & \PP^1\\ & Y\ar[ur]_\Psi  &}\]
where $\Psi$ is a finite morphism. 

\begin{prop} We can describe the turbulent foliation $\scrF$ as follows:
\begin{enumerate}
\item There is a meromorphic differential $\eta\in \scrF^\perp$ such that 
\[ \eta = \tau^\ast \eta_Y + \eta_{\rm hol}, \]
where $\eta_Y$ is a meromorphic differential on the elliptic curve $Y$ and $\eta_{\rm hol}$ is a holomorphic differential on $A$.
\item Let $\underline{\omega}_A$ be the $k$-vector space of holomorphic differentials on $A$, identified with the cotangent space of $A$ at the origin $0$; similarly for $E$ and $Y$. There is an exact sequence 
\[\xymatrix{0 \ar[r]& \underline{\omega}_Y\ar[r]& \underline{\omega}_A\ar[r] & \underline{\omega}_E \ar[r] & 0.}\]
\item Choose a nonzero holomorphic differential $\omega_Y$ on $Y$, viewed as a holomorphic nonzero differential $\tilde{\omega}_Y$ on $A$, and a nonzero holomorphic differential $\tilde{\omega}_E$ on $A$, pulling back to a nonzero holomorphic differential $\omega_E$ on $E$. Then we can write $\scrF$ uniquely as 
\[ \scrF = \eta^\perp, \qquad \eta = \tau^\ast(f\omega_Y) + \tilde{\omega}_E, \]
where $f$ is a non-constant function on $Y$.
\end{enumerate}
\end{prop}
\begin{proof} In essence, this proof is the one in \cite{Brunella}. The exact sequence of abelian varieties
\[ \xymatrix{0 \ar[r] & E \ar[r]^\iota &A \ar[r]^\tau & Y \ar[r] & 0}\]
yields a sequence of Lie algebras
\[ \xymatrix{0 \ar[r] & \Lie(E) \ar[r] & \Lie(A) \ar[r] & \Lie(Y) \ar[r] & 0.}\]
As $\iota$ is a closed immersion, the map $\Lie(E) \arr \Lie(A)$ is injective, and as $\tau$ is a smooth morphism, the map $\Lie(A) \arr \Lie(Y)$ is surjective. Comparison of dimensions implies the sequence $ 0 \arr \Lie(E) \arr \Lie(A) \arr \Lie(Y) \arr 0$ is exact, and so is the dual sequence 
\begin{equation}\label{eqn: exact on diff} 
\xymatrix{0 \ar[r] & \underline{\omega}_Y \ar[r] & \underline{\omega}_A \ar[r] & \underline{\omega}_E \ar[r] & 0.} 
\end{equation}
Using again that $\tau$ is a smooth morphism we have an exact sequence of free sheaves:
\begin{equation}
     \xymatrix{0 \ar[r] & \tau^\ast \Omega^1_Y \ar[r] & \Omega^1_A \ar[r] & \Omega^1_{A/Y} \ar[r] & 0.}
     \end{equation}
We claim that we can identify $\Omega^1_{A/Y}$ with $\underline{\omega}_E \otimes_k \calO_A$. To do that, choose a non-vanishing global regular differential $\omega_Y$ on $Y$ and view it as a differential $\tilde{\omega}_Y$ on $A$ via the map $\tau^\ast$. Let $\tilde{\omega}_E$ be a global regular differential on $A$ such that $\omega_E:= \iota^\ast \tilde{\omega}_E$ is a nonzero differential on $E$. Then, $\Omega^1_A = \calO_A \tilde\omega_Y \oplus \calO_A \tilde\omega_E$ and that can also be written as $\left(\underline{\omega}_Y \otimes_k \calO_A\right) \oplus \left(\underline{\omega}_E \otimes_k \calO_A\right)$, where the first summand is identified with $\tau^\ast \Omega^1_Y$ and the second with $\Omega^1_{A/Y}$.  To summarize, we have a split exact sequence
 \begin{equation}\label{equ: exact on sheaves of diff} 
 \xymatrix{0 \ar[r]& \tau^\ast \Omega^1_Y \ar[r]& \Omega^1_A \ar[r]& \underline{\omega}_{E}\otimes_k \calO_A \ar[r]& 0.}
 \end{equation}
 Of course, taking its global sections gives again (\ref{eqn: exact on diff}). 

 On the other hand, we have the relation 
 \[ N_\scrF = \tau^\ast N_Y,\]
 where $N_Y = \Psi^\ast \calO_{\PP^1}(1)$ (it is an ample line bundle on $Y$ that Brunella calls the normal reduction of $N_\scrF$). Twist the exact sequence (\ref{equ: exact on sheaves of diff}) by tensoring it with the locally free sheaf $N_\scrF$ to get a split exact sequence
  \begin{equation}\label{equ: twisted exact on sheaves of diff} 
  \xymatrix{0 \ar[r]& \tau^\ast \Omega^1_Y\otimes \tau^\ast N_Y \ar[r]& \Omega^1_A\otimes \tau^\ast N_Y \ar[r]& \underline{\omega}_{E}\otimes_k \tau^\ast N_Y \ar[r]& 0.}
 \end{equation}
Now, for any coherent sheaf $L$ on $Y$, the projection formula gives $H^0(A, \tau^\ast L) \cong H^0(Y, \tau_\ast\tau^\ast L)$. However, as the fibers of $\tau$ are connected and projective (all isomorphic to $E$, in fact) $\tau_\ast\calO_A = \calO_Y$ and so if $L$ is locally free, $\tau_\ast \tau^\ast L \cong L$ and we find that $H^0(A, \tau^\ast L) \cong H^0(Y, L)$. Let us then take global sections of (\ref{equ: twisted exact on sheaves of diff}), using that the sequence splits, to find an exact sequence 
\[ \xymatrix{0 \ar[r] & H^0(A, \tau^\ast (\Omega^1_Y \otimes N_Y)) \ar[r]\ar@{=}[d] &  H^0(A, \Omega^1_A \otimes N_\scrF) \ar[r] &  \underline{\omega}_{E}\otimes_k H^0(A, \tau^\ast N_Y)\ar@{=}[d] \ar[r] &  0\\
 & H^0(Y, \Omega^1_Y \otimes N_Y) & & \underline{\omega}_{E}\otimes_k H^0(Y, N_Y) &
} \]
That means that any global section in $H^0(A, \Omega^1_A \otimes N_\scrF)$ can be written as $a \tilde\omega_Y + b \tilde\omega_E$, with $a, b \in H^0(Y, N_Y)$ (identifying $a$ with $\tau^\ast a$, etc.). The point of all this is that to give a foliation $\scrF \injects \calT_A$ is the same as giving the surjection $\calT_A \arr N_\scrF$, which is the same as giving the injection $N_\scrF^\ast \injects \Omega^1_A$ (cf. (\ref{eqn: foliation ses}), (\ref{eqn: dual to foliation ses})). That, in turn, is equivalent to giving a nowhere vanishing section of $H^0(A, \Omega^1_A \otimes N_\scrF)$, which has the form $a \tilde\omega_Y + b \tilde\omega_E$. Returning to meromorphic differentials, this means that the foliation $\scrF$ is defined by the meromorphic differential 
\[ \eta = \tau^\ast (f \omega_Y) + \tilde\omega_E,\]
where $f$ is a function such that $fb = a$ (such a function exists as the turbulent condition implies $b\neq 0$), and writing $N_Y = \calO_Y(\delta)$ one has $(f)_\infty = \delta$.
\end{proof}

\subsection{Classification of turbulent foliations}
\begin{thm}\label{thm: classfication of foliations} Smooth turbulent foliations on an abelian surface $A$ are in $1:1$ correspondence with the data $(E, f)$, where:
\begin{enumerate}
\item $E$ is an elliptic curve $E \subseteq A$ (passing through the origin);
\item $f$ is a non-constant meromorphic function on the elliptic curve $Y = A/E$.
\end{enumerate}
\end{thm}
\begin{proof} Given a foliation $\scrF$, with classifying morphism $\Phi_\scrF$, we defined the elliptic curve $E$ as the connected component (with the induced reduced structure) of $\Phi_\scrF^{-1}(\Phi_\scrF(0))$ that contains the origin $0 \in A$. Then,
letting $Y = A/E$ and $\tau\colon A \arr Y$ be the canonical morphism, one chooses differentials $\tilde{\omega}_E, \tilde{\omega}_Y$ as above to find that \[\scrF = \eta^\perp, \qquad \eta = \tau^\ast(f\omega_Y) + \tilde{\omega}_E.\]

Note that having chosen the differentials (an arbitrary choice that is made and fixed once $E$ is given),~$\eta$ is uniquely determined and $\eta$ determines $f$ uniquely.

Conversely, given the data $(E, f)$, choose differentials $\tilde{\omega}_E, \tilde{\omega}_Y$. Let $\eta = \tau^\ast(f\omega_Y) + \tilde{\omega}_E$ and $\scrF = \eta^\perp$. Then $\scrF$ is a rank $1$ saturated subsheaf of $\calT_A$ and, by Lemma~\ref{lem: reflexive and loc free}, $\scrF$ is locally free. To see it is a smooth foliation, note that $\eta$ is nowhere vanishing, but acquires poles along the pullback of the polar part $(f)_\infty$ of $f$. These are clearly removable singularities. At every point $x\in A$ such that $f$ has a pole at $\tau(x)$ we can find a function $g_x\in k(Y)$ (that, in fact, depends only on $\tau(x)$) such that $g_xf$ is regular and non-vanishing at $\tau(x)$. We may replace $\eta$ by $\tau^\ast(g_xf\omega_Y) + \tau^\ast(g_x)\tilde{\omega}_E$, which defines the same foliation but is given locally at $x$ by a regular non-vanishing differential. Thus,~$\scrF$ is a smooth foliation. 
\end{proof}

\medskip

\id Theorem~\ref{thm: classfication of foliations} motivates the following construction. Given an elliptic curve $Y$ over $k$ and a positive integer $d\geq 2$, let $\WW_{Y, d}$ be the set of pairs $(\uy, f)$, where $\uy$ is an effective divisor of degree $d$ on $Y$ and $f\in H^0(Y, \calO_Y(\uy))$. Further, we let $\WW^\circ_{Y, d}$ be the set of pairs $(\uy, f)$ such that $(f)_\infty = \uy$. 
We will later give a scheme structure to $\WW^\circ_{Y, d}$ and study its properties--see \S~\ref{sec: WW}, but for now we only make some initial observations. 
\begin{enumerate}
\item The divisors $\uy$ correspond to $k$-rational points of ${\rm Sym}^d Y: = Y^d/\Sigma_d$, where $\Sigma_d$ is the symmetric group on $d$ letters, acting by permuting the coordinates on $Y^d$.
\item Given an effective divisor $\uy$ of degree $d\geq 2$, the Riemann-Roch theorem implies that $H^0(Y, \calO_Y(\uy))$ is a $d$-dimensional vector space. (This holds also for $d=1$.)
\item The pairs $(\uy, f)$ parameterized by $\WW^\circ_{Y, d}$ form, for every $\uy$, an open dense subset of $H^0(Y, \calO_Y(\uy))$. (This fails for $d=1$.)
\end{enumerate}

\id We can rephrase Theorem~\ref{thm: classfication of foliations} as follows. 
\begin{thm}\label{thm classtrbfol} Turbulent foliations $\scrF$ on $A$ are parameterized by the data $(E, (\uy, f))$ constructed as follows. The elliptic curves $E, Y$ and the canonical map $\tau\colon A \arr Y$ are as before, and $(\uy, f)$ is a point of $\WW^\circ_{Y, d}$. Moreover, in this case $\scrF \cong \calO_A(-D)$, where $D$ is the divisor on $A$ that is the pullback of the polar divisor of $f$; to wit,
\[ D = \tau^\ast((f)_\infty).\]
\end{thm}

\subsection{$E$ is a discrete invariant}\label{sec discreteinv} We have seen that a turbulent foliation $\scrF$ has an elliptic curve $E$ that lies on $A$ and passes through the origin. Our purpose here is to show that this is a discrete invariant in the sense that there is no continuous family of elliptic curves all passing through the origin. Before discussing that, we make two observations:
\begin{enumerate}
    \item Given an elliptic curve $E$ on $A$, let $Y = A/E$ and choose an isogeny $E \times Y \arr A$, such that its restriction to $E \times \{0_Y\}$ is the inclusion $E \injects A$ (this is possible, as we shall explain shortly). The family of curves $(E\times \{y\})_{y\in Y}$ on $E\times Y$ defines a family of genus $1$ curves on $A$ (all isomorphic to $E$ as genus $1$ curves) and this family covers $A$, but only finitely many of its members pass through $0_A$.
    \item For curves of higher genus the analogous statement is false. Let $C$ be a smooth projective curve of genus $2$ and consider $C - C: = \{[P]-[Q] : P, Q \in C\} = \{C - [Q]\}_{Q\in C}$ as a family of curves indexed by $C$ on the abelian surface $A = {\rm Jac}(C) = {\rm Div}^0(C)/{\rm Prin}(C)$. For a fixed $Q$, the map $C \arr A, P \mapsto [P]-[Q]$ is a closed immersion. Every curve $C - [Q]$ passes through $0_A$ and every point on $A$ can be written as $[P]-[Q]$ for some $P, Q \in C$. 
    
    To see the last point 
    use the well-known fact that for any point $W\in C$ every point on $A$ can be written as $[R]+[S]-2[W]$ for some points $R, S \in C$. Choose $W$ to be a Weierstrass point. If $\iota$ denotes the hyperelliptic involution, we have $\iota(W) = W$. Now, $[R]+[S]-2[W] = [R]-[\iota(S)]$, because $[S]+[\iota(S)] - 2[W]$ is a principal divisor.

    Thus, $A$ is swept by a $1$-dimensional family of genus $2$ curves, all passing through its origin. 
\end{enumerate}

\vspace{0.5cm}

\id Regarding the non-existence of continuous families, we follow \cite[Lecture 22]{MumfordCurvesOnSurfaces}. A {\bf family of curves} on a smooth projective surface $A$ over an algebraically closed field $k$, parameterized by a $k$-scheme $S$, is a relatively effective Cartier divisor  $\scrD \subset A \times S$ over $S$ (loc. cit., Lecture 10, p. 72 and Lecture 12, p. 88). 

Given a closed point $s\in S$, denote by $\scrD_s$ the fiber of $\scrD$ over the point $s\in S$, which is an effective curve on $A$. There is a canonical linear homomorphism--the {\bf characteristic map}--
\[ \rho\colon \calT_{S, s} \arr H^0(A, N_{\scrD_s}), \]
where $\calT_{S, s}$ is the tangent space of $S$ at the point $s$ and $N_{\scrD_s}$ is the normal bundle to $\scrD_s$ in the surface $A$; furthermore, for the universal family $\scrD \subset A \times C(\xi), \rho$ is an isomorphism (see loc. cit., Lecture 12, pp. 88-89 and Lecture 15 for the notion of such a universal family over a base $C(\xi)$ (really, a Hilbert scheme) for a given numerical class $\xi$). The normal bundle $N_{\scrD_s}$ is a line bundle on $A$ (or, by a common abuse of notation, on $\scrD_s$) that sits in an exact sequence
\[ 0 \arr \calT_{\scrD_s} \arr \calT_A\otimes \calO_{\scrD_s} \arr N_{\scrD_s} \arr 0.\]
In fact, it has multiple equivalent definitions, one of which is that it sits in an exact sequence
\[ 0 \arr \calO_A \arr \calO_A(\scrD_s) \arr N_{\scrD_s} \arr 0\]
(which is how it is defined in \cite[Lecture 22, p. 153]{MumfordCurvesOnSurfaces}).
Coming back to our situation, where $A$ is an abelian surface and $E$ is an elliptic curve on it, let $\xi$ be the numerical class of the divisor $E$. Then the dimension of $C(\xi)$ at the point $s$ corresponding to $E$ is bounded above by $\dim H^0(A, N_{\scrD_s}) = \dim H^0(E, N_{\scrD_s}\vert_E) = 1$ (as $N_{\scrD_s}$ is supported on $E$, and is in fact the trivial sheaf on $E$). On the other hand, we have the family (of algebraically equivalent, hence numerically equivalent, divisors) parametrized by $S = A$ itself, whose fiber over $t\in A$ is $E + t$. This shows that the only family of curves extending $E$ is this family of translations of $E$. In particular, if members of this family contain $0_A\in A$ then they are equal to $E$. The discussion in fact shows that the component of $C(\xi)$ passing through $E$ is $A/E = Y$.

Thus, every algebraic family of elliptic curves in the surface $A$ for which every member passes through $0_A$ is a constant family. 

For completeness, remark also that the isogeny class of elliptic curves on $A$ is very restricted. First, note the following fact. \begin{lem}
    If $f\colon A \arr B$ is an isogeny, then $E \mapsto f(E)$ is a bijective map between the set of elliptic curves on~$A$ and the set of elliptic curves on $B$.
\end{lem}
\begin{proof} As $f$ is an isogeny, $f(E)$ is an elliptic curve whenever $E$ is. Note also that $E$ is the connected component (with its reduced induced structure) through the identity of $f^{-1}(f(E))$, hence $E \mapsto f(E)$ is injective. Similarly, if $D$ is an elliptic curve on $B$, let $E$ be the connected component (with its reduced induced structure) through the identity of $f^{-1}(D)$. Then $f(E) = D$.
\end{proof}
Thus, to understand elliptic curves on $A \sim E_1 \times E_2$ we may reduce to the case of $ E_1 \times E_2$. If $E_1$ is not isogenous to $E_2$ then every elliptic curve on $E_1\times E_2$ is either $E_1\times \{0\}$ or $\{0\} \times E_2$, because an elliptic curve that does not project to $0$ under either of the two projections implies an isogeny between $E_1$ and $E_2$. Thus, also $A$ has precisely two elliptic curves on it. 

If, on the other hand, $E_1$ is isogenous to $E_2$, say $f\colon E_1 \arr E_2$ is an isogeny, then the graphs of $f\circ [m]$, $m=1, 2, \dots$ provide infinitely many distinct elliptic curves on $E_1 \times E_2$ and so $A$ has infinitely many elliptic curves as well. 

\

\subsection{The Hilbert polynomial}\label{sec Hilbertpol} We know now that a turbulent foliation $\scrF$ on $A$ has associated to it invariants $(E, f)$ and one may view the exact sequence $\xymatrix@C=0.5cm{0 \ar[r] & \scrF \ar[r] & \calT_A \ar[r]& N_\scrF \ar[r]& 0}$ as the exact sequence of vector bundles
\[\xymatrix{0 \ar[r] & \calO_A(-D) \ar[r] & \calT_A \ar[r]& \calO_A(D) \ar[r]& 0,}\]
where $D = \tau^{-1}((f)_\infty)$. 
Our purpose in this section is to compute the Hilbert polynomial of $\calO_A(D)$ relative to a suitable ample line bundle. 

To fix ideas, consider the following: \textit{Let $\tau\colon A \arr Y$ be a surjective homomorphism of $A$ onto an elliptic curve $Y$ with connected kernel $E$. Let $\delta> 0$ be an effective divisor on $Y$ and let $D = \tau^\ast(\delta)$. Calculate the Hilbert polynomial of $\calO_A(D)$ relative to a well-chosen polarization on $A$. }

\begin{lem} \label{Lemma:product up to isogeny} One can construct a diagram of the form
\[ \xymatrix{E \times Y \ar[r]^>>>>>g &  A \ar[r]^<<<<\gamma&  E\times Y}\]
such that $\gamma \circ g = [\deg(g)]$ and $g\vert_{E\times \{ 0_Y\}}$ is the inclusion map of $E$ in $A$. Moreover, $\deg(g)=n^2$ where $n$ is the order of the extension class $0 \arr E \arr A \arr Y \arr 0$ in ${\rm Ext}^1(Y, E)$.
\end{lem}
\begin{proof} Poincar\'e's theorem about the semi-simplicity of the category of abelian varieties up to isogeny implies that the extension $\zeta\colon 0 \arr E \arr A \arr Y \arr 0$ defines a torsion class in $\Ext^1(Y, E)$ (in the category of commutative group schemes over $k$). So, for some positive integer $n$ the class $n\zeta = 0$, and so we have a commutative diagram,
\[ \xymatrix{n\zeta: & 0 \ar[r] & E \ar[r]\ar[d]^= & X \ar[r]\ar[d] & Y \ar[r]\ar[d]^{[n]_Y} & 0\\
\zeta:& 0 \ar[r] & E \ar@{^{(}->}[r] & A \ar[r]^\tau & Y \ar[r] & 0,}\]
where the right square is cartesian and $n\zeta = 0$. Thus, the upper row splits. Choosing a splitting, the vertical arrow becomes an isogeny $g\colon X = E \times Y \arr A$ and the composition $E \arr X \arr A$ is indeed the inclusion $E \arr A$. Note that $\deg(g) = \deg([n]_Y\colon Y \arr Y) = n^2$, where $n = \ord(\zeta)$.
\end{proof}
\begin{rmk} The proof shows that we can take $\gamma$ such that $\gamma \circ g = [n]_{E\times Y}$ instead of  $\gamma \circ g = [n^2]_{E\times Y}$, but this is really of no importance to the rest of the paper. 
\end{rmk}

Using the lemma, we endow $A$ with the polarization $\gamma^\ast (\lambda_E \times \lambda_Y)$. This is not a canonical choice, but one that can be made once and for all given the elliptic curve $E$ and, in particular, it is independent of the foliation $\scrF$ as long as $E_\scrF = E$. Furthermore, the calculation of the Hilbert polynomials will show they are independent of this choice.

\medskip

\id Recall the following general property concerning Hilbert polynomials (cf. \cite[(9.26)]{EvdGM}): Let $f\colon X \arr Y$ be an isogeny of abelian varieties. Choose an ample line bundle $H$ on $Y$. Then the Hilbert polynomial of $f^\ast L$ with
respect  to the ample bundle $f^\ast H$ is just $\deg(f)$ times the Hilbert polynomial of $L$ with respect
to $H$. 

Based on this property, we want to use the following diagram to reduce the calculation of the Hilbert polynomial of $\calO_A(D)$ relative to the polarization $\gamma^\ast (\lambda_E \times \lambda_Y)$ to the calculation of $\calO_X(D^\prime)$ where $X = E \times Y$ is endowed with a multiple of the principal polarization $\lambda_E \times \lambda_Y$, and $D^\prime$ is a certain divisor depending on $D$ and the maps. We use the existence of the commutative diagram 
\[ \xymatrix{X = E\times Y \ar[d]^g\ar[rd]^{\tau \circ g} \ar[r]^{\pi_Y}  & Y\ar[d]^\beta\\ A\ar[r]_\tau & Y = A/E }\]
in which $\beta=[n]_Y$ is multiplication by $n$ on $Y$.

The map $g$ is an isogeny of polarized abelian varieties because $g^\ast(\gamma^\ast (\lambda_E \times \lambda_Y)) = [n^2]^\ast(\lambda_E \times \lambda_Y) = n^4 (\lambda_E \times \lambda_Y)$. Furthermore, as the divisor $D$ is of the form $\tau^\ast(\delta)$, where $\delta$ is an effective divisor of degree $d$ on $Y$, one has
\[ g^\ast \calO_A(D) = g^\ast \tau^\ast \calO_Y(\delta) = \pi_Y^\ast \beta^\ast \calO_Y(\delta) =  \pi_Y^\ast  \calO_Y(\delta^\prime) = \calO_X(D^\prime), \]
where
\[ \delta^\prime = [n]_Y^\ast(\delta), \qquad D^\prime = E \times \delta^\prime.\]
We therefore consider first the following:

\

\id {\bf Question}: \textit{Equip $X = E \times Y$ with the product polarization and let $\delta^\prime$ be an effective divisor on $Y$. Let $\Gamma = E \times \delta^\prime$ be the fibral divisor on $X$ obtained as the pullback of $\delta^\prime$. Calculate the Hilbert polynomial of $\calO_X(\Gamma)$ relative to $\lambda_E \times \lambda_Y$. }

\

\id Let $\calL_Y = \calO_Y([0_Y]), \calL_E = \calO_E([0_E])$ and 
\[ \calL = \calL_E \boxtimes \calL_Y.\]
This is a line bundle on $X$ defining the product polarization. Let $M = \calO_X(\Gamma)$. For the Hilbert polynomial we need to calculate the function
\[ a \mapsto \chi(M \otimes \calL^a), \qquad a\gg 0.\]
But, for $a>0$,
\begin{equation*}
\begin{split} \chi(M \otimes \calL^a) & =   \chi\left(\calO_E(a[0_E]) \boxtimes \calO_Y(\delta^\prime + a[0_Y])\right)\\ & = 
  \chi(\calO_E(a[0_E])) \cdot \chi(\calO_Y(\delta^\prime + a[0_Y])) \\ & =  a(\deg(\delta^\prime) + a).
\end{split}
\end{equation*}
This gives us that the Hilbert polynomial of $M$ relative to $\lambda_E \times \lambda_Y$ is 
\[ t(t + \deg(\delta^\prime)).\]
The Hilbert polynomial of $M$ relative to $n^4(\lambda_E \times \lambda_Y)$ is obtained by replacing $t$ by $n^4 t$ to get
\[ n^4t(n^4t + \deg(\delta^\prime)).\]
If we go back to the original setting, we have $\deg(\delta^\prime) = \deg(\beta)\cdot \deg(\delta)=n^2\deg(\delta)$ and the Euler characteristic is scaled by $\deg(g) = \deg(\beta)=n^2$. 

\medskip

\id {\bf Conclusion:} \textit{The Hilbert polynomial $P(t)$ of $\calO_A(D)$ relative to the polarization $\gamma^\ast(\lambda_E \times \lambda_Y)$ is the polynomial $ \frac{1}{\deg(\beta)}(n^8t^2 + n^4 \deg(\beta) \deg(\delta) t)$. That is,
\begin{equation}\label{eqn:P}P(t)  = n^6t^2 + n^4\deg(\delta) t.\end{equation}}

\medskip

 \id \textit{We observe that the computation above applies in fact to any divisor $\delta$ on $Y$. 
 Therefore, the Hilbert polynomial $P^\prime$ of $\calO_A(-D)$ is 
 \begin{equation}\label{eqn:P'} P'(t) =  n^6t^2 - n^4\deg(\delta) t. \end{equation}}

\medskip

\id As we shall discuss soon in greater depth, this has the following consequence. Consider a geometric point $(\uy, f)$ of $\WW^\circ_{Y, d}$. It defines for us a foliation on $A$ (in the sense of a smooth involutive subsheaf), $0 \arr \calO_A(-D) \arr \calT_A \arr \calO_A(D) \arr 0$ 
such that the Hilbert polynomials of $\calO_A(D)$ and $\calO_A(-D)$ are independent of the point and are given by $P$ and $P^\prime$ of Equations~(\ref{eqn:P}) and (\ref{eqn:P'}), respectively. We thus get a function 
\[ \WW^\circ_{Y, d} (k) \arr {\rm Quot}^{P}(k),\]
where ${\rm Quot}^{P}$ is a certain Quot scheme to be introduced later.

\

\subsection{The Cartier operator}\label{subsec:Cartier} The Cartier operator $\scrC$, defined for any non-singular variety~$X$ over an algebraically closed field $k$ of characteristic $p$, with function field $K = k(X)$, is an operation on closed differentials (and higher differentials). However, at the generic point, and specializing to $1$-forms, and to $\Omega^1_{X/k}$, it is a map
\[ \scrC\colon Z^1_{K} \arr \Omega^1_{K}, \]
where $Z^1_K$ denotes the space of closed meromorphic differentials (i.e. elements $\omega$ of $\Omega^1_K$ such that $d\omega = 0$). The operator has the following properties \cite{Seshadri}:
\textit{\begin{itemize}
\item $\scrC (\omega_1 + \omega_2) = \scrC(\omega_1) + \scrC(\omega_2)$.
\item $\scrC(f^p\omega) = f\scrC(\omega)$, $f\in K$.
\item $\scrC(dx) = 0, \forall x \in K$. That is, $\scrC$ vanishes on $B^1_K$--the exact differentials. 
\item $\scrC$ induces a $\sigma^{-1}$-linear isomorphism $Z^1_K/B^1_K \arr \Omega^1_K$ as vector spaces over $K^p$, where $\sigma(f) = f^p$ is the Frobenius homomorphism. 
\item Let $x_1, \dots, x_n$ be a $p$-basis for $K$. That means that the monomials 
\[ x_1^{e_1} x_2^{e_2}\cdots x_n^{e_n}, \quad \forall i, 0 \leq e_i < p, \]
form a basis for $K$ as a vector space over $K^p$. Such a $p$-basis always exists. Then
\[ \scrC(x_i^{p-1} dx_i) = dx_i, \quad \forall i.\]
\item More generally, as one can ``integrate'' differentials of the form $\ell^p x_i^j dx_i$ for $\ell\in K$ and $0\leq j < p-1$, any $\omega\in Z^1_K$ can be written uniquely in the form\[\omega = df + \sum_{i=1}^n \ell_i^p x_i^{p-1} dx_i,\] for suitable $f$ and $\ell_i$ in $K$, where $f$ is unique modulo $K^p$. Then
\[ \scrC(\omega) = \sum_{i=1}^n \ell_i dx_i.\]
\item In fact, for any nonzero $x \in K$
\[ \scrC(x^{p-1}dx) = dx, \quad \scrC(x^{-1}dx) = x^{-1} dx.\]
\item For $\omega \in Z^1_K$ and $D$ a derivation of $K$ (which is automatically $K^p$-linear and $K^p$ is contained in its constants) we have
\begin{equation}\label{formula: cartier}\langle D, \scrC(\omega)\rangle^p = \langle D^p, \omega\rangle - D^{p-1}(\langle D, \omega\rangle).
\end{equation}
\end{itemize}}

\

We include a general functoriality statement for the Cartier operator, for lack of an adequate reference. 
\begin{lem}\label{functoriality of Cartier}
    Let $k$ be a perfect field of characteristic $p$ and let $\varphi\colon X' \to X$ be a
morphism of smooth $k$-varieties. For every closed regular differential $\omega$ on $X$,
\[ \mathscr{C}_{X'}(\varphi^\ast\omega) = \varphi^\ast\big(\mathscr{C}_X(\omega)\big). \]
The same holds for closed meromorphic forms, provided $\varphi$ is dominant.
\end{lem}
\begin{proof}
   Let $Z^i_X = \ker(d\colon \Omega^i_X \arr \Omega^{i+1}_X)$ be the subsheaf of closed forms and $B^i_X = d\Omega^{i-1}_X$ be the subsheaf of exact forms. Cartier proved that 
   \[\scrC_X\colon Z^i_X/B^i_X \arr\Omega^i_X \] is an isomorphism. Let $\scrC_X^{-1}$ be the inverse isomorphism. It has the properties
   \[ \scrC_X^{-1}(f) = f^p, \qquad \scrC_X^{-1}(dg) = g^{p-1} dg \pmod{B^i_X}.\]
   Furthermore, it is additive and multiplicative relative to $\wedge$ and so, in particular, 
   \[ \scrC_X^{-1}(fdg_1 \wedge dg_2 \wedge \cdots \wedge dg_i) = f^p \prod_{n=1}^i g_n^{p-1}\cdot  dg_1 \wedge dg_2 \wedge \cdots \wedge dg_i. \]
   In fact these properties characterize $\scrC_X^{-1}$ \cite{Cartier, Seshadri}. 

   Note that since $\varphi^\ast$ commutes with $d$, it induces maps $\Omega^i_X \arr \Omega^i_{X'}$, taking closed (resp. exact) forms to closed (resp. exact) forms. It is thus enough to show that
   \[ \scrC_{X'}^{-1} \circ \varphi^\ast = \varphi^\ast \circ \scrC_X^{-1}.\]
   Using that $\Omega^i_X$ is generated by the forms $fdg_1 \wedge dg_2 \wedge \cdots \wedge dg_i$, it is enough to show this relation on such forms. However, 
   \begin{equation*}
   \begin{split} \scrC_{X'}^{-1} \circ \varphi^\ast (fdg_1 \wedge dg_2 \wedge \cdots \wedge dg_i) &= \scrC_{X'}^{-1} (\varphi^\ast(f) d\varphi^\ast(g_1) \wedge d \varphi^\ast(g_2) \wedge \cdots \wedge d\varphi^\ast(g_i)) \\
   & = \varphi^\ast(f)^p (\prod_{n=1}^i \varphi^\ast(g_n)^{p-1}) \cdot d\varphi^\ast(g_1) \wedge d \varphi^\ast(g_2) \wedge \cdots \wedge d\varphi^\ast(g_i)\\ & =
   \varphi^\ast \left(f^p \prod_{n=1}^i g_n^{p-1}\cdot dg_1 \wedge dg_2 \wedge \cdots \wedge dg_i\right)\\ 
   & = \varphi^\ast \scrC_X^{-1}(f dg_1 \wedge dg_2 \wedge \cdots \wedge dg_i).
   \end{split}
   \end{equation*}
   The same argument works for meromorphic forms, where the assumption that the morphism is dominant is used to show that the map $\varphi^\ast$ is well-defined. 
\end{proof}

\ 

\begin{exa}\label{exa: cartier on curve}
Assume now that $X$ is a curve. Then any meromorphic differential is closed and the Cartier operator is a $\sigma^{-1}$-linear isomorphism
\[ \scrC\colon \Omega^1_K/dK \overset{\sim}{\Arr} \Omega^1_K.\]
If $x$ is a $p$-basis, any meromorphic differential $\omega$ can be written as
\[ \omega = f dx =(g_0^p + g_1^px + \cdots+ g_{p-1}^p x^{p-1})dx  = dg +  g_{p-1}^p x^{p-1}dx, \]
for some uniquely determined functions $g_i\in K$, and a function $g\in K$ (not unique; one can add to it any element from $K^p$).
Then,
\[ \scrC(\omega) = g_{p-1}dx.\] 
\end{exa}

 \

\subsection{$p$-closed foliations}\label{567} Let $X$ be a smooth variety over a field $k$ of characteristic $p$ and let $\scrF \subseteq \calT_X$ be a smooth foliation. Given a section $D \in \scrF(U)$ over an open set $U \subseteq X$, we view $D$ as a derivation on $\calO_X(U)$. Then $D^p$, the $p$-fold iteration of $D$, is also a derivation and as such belongs to $\calT_X(U)$. We say that $\scrF$ is {\bf $p$-closed} if $D^p \in \scrF(U)$ too, for all $U$ and $D\in \scrF(U)$. Both the involutive and the $p$-closed conditions can be phrased as follows (cf. \cite[Lemma 4.2]{Ekedahl}): 
\begin{itemize}
    \item 
A subsheaf $\scrF \subset \calT_X$ is involutive if and only if the morphism of $\calO_X$-modules induced by the Lie bracket, $\bigwedge^2 \scrF \arr \calT_X/\scrF$, is the zero morphism. Therefore, this condition is automatic if $\scrF$ is an invertible sheaf;
\item if $\scrF$ is involutive, then it is $p$-closed if and only if the $p$-th power morphism of $\calO_X$-modules $F_{\rm abs}^\ast \scrF \arr \calT_X/\scrF$ is the zero morphism, where $F_{\rm abs}$ is the absolute Frobenius morphism of $X$. 
\end{itemize}

\medskip

\id Let $A$ be an abelian surface over an algebraically closed field $k$ of characteristic $p>0$ and $E \subset A$ an elliptic curve. Then, as above, we have the exact sequence $0 \arr E \arr A\overset{\tau}{\arr} Y:=A/E \arr 0$ and a choice of differentials $\tilde \omega_Y,\tilde \omega_E$ related to differentials $\omega_Y, \omega_E$ on $Y$ and $E$, respectively. If $Y$ is an ordinary elliptic curve, we will henceforth choose $\omega_Y$ such that $\scrC(\omega_Y) = \omega_Y$ and if $Y$ is a supersingular elliptic curve any choice of $\omega_Y$ will do, and for every such choice $\scrC_Y(\omega_Y) = 0$.

Any turbulent foliation $\scrF$ with an elliptic curve invariant $E_\scrF = E$ is defined  by a meromorphic differential form
\[ \eta = \tau^\ast(f) \tilde \omega_Y + \tilde \omega_E, \]
where $f$ is a uniquely determined non-constant meromorphic function on $Y$. 

\begin{prop}\label{prop:p closed in terms of C} \begin{enumerate}
    \item The form $\eta$ is closed. 
     \item The foliation $\scrF$ is $p$-closed if and only if $\scrC_A(\eta)= \lambda \eta$ for some meromorphic function $\lambda$.
\end{enumerate}
\end{prop}

\begin{proof} We begin with the first claim. We will need that regular differentials on $A$ are closed. This is well known, but we include the argument for the convenience of the reader. Let $\omega$ be a regular differential on $A$. As is well known, $\omega$ is then translation invariant (this follows from the fact that $\Omega^1_A \cong \calO_A^2$, or more canonically $\underline\omega_A \otimes_k \calO_A$).  We want to prove that $d\omega = 0$. Using the triviality of $\calT_A$ and $\Omega^1_A$, it is enough to prove that for any two invariant vector fields $D_1, D_2$ we have 
\[d\omega (D_1, D_2) = 0\]
(thinking of $d\omega$ as an alternating bilinear form on $\calT_A$). We have
\[ d\omega (D_1, D_2) = D_1(\omega(D_2)) - D_2(\omega(D_1)) -\omega([D_1, D_2]).\]
However, as $A$ is a commutative algebraic group $\Lie(A)$ is commutative and so $[D_1, D_2] = 0$. And, as $\omega$ and $D_i$ are translation-invariant, $\omega(D_i)$ is a constant. So all the terms on the right-hand side vanish and $d\omega(D_1, D_2) = 0$.

Now, we have 
\[ d\eta = d \left(\tau^\ast(f\omega_Y)\right) + d(\tilde\omega_E) = \tau^\ast (d(f\omega_Y)) + d(\tilde\omega_E), \]
where we used functoriality of $d$. Now, holomorphic differentials on $A$ are closed and so $d(\tilde\omega_E) = 0$. Secondly,  $d(f\omega_Y) = 0$ since $Y$ is one-dimensional. Thus, $d\eta = 0$.

\medskip
\id Part (2) follows from formula (\ref{formula: cartier}). Using $\omega=\eta$ and taking for $D$ a local basis of $\scrF$, we have $\langle D, \eta \rangle =0$ and the formula becomes $\langle D, \scrC(\eta)\rangle^p = \langle D^p, \eta\rangle$. Hence, $D^p$ is a local section of~$\scrF$ if and only if $\langle D^p, \eta\rangle=0$, which holds if and only if $\langle D, \scrC(\eta)\rangle=0$. The last condition holds if and only if $\scrC(\eta)=\lambda \eta$ for $\lambda$ a meromorphic function on $A$ (possibly $0$).
\end{proof}

\subsubsection{Extensions of elliptic curves}\label{sec:extension} The notion of extension of elliptic curves in the category of commutative algebraic groups has already arisen one way or another a few times. We provide here some general comments concerning those. We thank Joelle Lim \cite{Joelle} for enlightening discussions and permission to include here some results that will appear in her thesis. 

The extension group of two elliptic curves $E, Y$ over an algebraically closed field $k$ (to fix ideas) in the abelian category of commutative algebraic groups can be defined by means of Yoneda extensions. See \cite{OortCGS} for this, as well as general background results. It satisfies familiar properties of $\Ext$ functors that are well known for abelian categories with enough injectives. In particular, given a short exact sequence of commutative algebraic groups 
\[ \xymatrix{0 \ar[r]& A \ar[r]& B \ar[r] &C \ar[r] &0}\] and a commutative algebraic group $Y$, there are long exact sequences of abelian groups 
\[\xymatrix@C=0.5cm{0 \ar[r]& \Hom(Y, A) \ar[r]& \Hom(Y, B) \ar[r] &\Hom(Y, C) \ar[r] & \Ext^1(Y,A) \ar[r]& \Ext^1(Y, B)\ar[r]& \Ext^1(Y, C), }\]
as well as 
\[ \xymatrix@C=0.5cm{ 0\ar[r]&\Hom(C, Y)\ar[r]&\Hom(B, Y)\ar[r]&\Hom(A, Y)\ar[r]& \Ext^1(C, Y)\ar[r]&\Ext^1(B, Y)\ar[r]& \Ext^1(A, Y)}.\]
In addition, $\Ext^1(Y, A)$ classifies equivalence classes of extensions 
\[ \xymatrix{0 \ar[r]& A \ar[r]& B \ar[r] &Y \ar[r] &0}\]
in the usual way. 

Applying it to our situation $0 \arr E \arr A \arr Y\arr 0$, where $E$ and $Y$ are elliptic curves, one knows that $\Ext^1(Y, E)$ is a torsion group. 
Take a positive integer $n$. The short exact sequence $0 \arr Y[n] \arr Y \arr Y \arr 0$ gives
\[\xymatrix@C=0.45cm{0 \ar[r]& \Hom(Y, E) \ar[r]^{[n]}&\Hom(Y, E) \ar[r]&\Hom(Y[n], E)\ar[r]&\Ext^1(Y, E) \ar[r]^{[n]}&\Ext^1(Y, E) \ar[r]&\Ext^1(Y[n], E),}\]
from which we get
\begin{equation}
     \xymatrix{0 \ar[r]& \Hom(Y, E)/n\Hom(Y, E) \ar[r]& \Hom(Y[n], E)\ar[r]& \Ext^1(Y, E)[n] \ar[r]& 0.}
     \end{equation}
On this basis we can perform the following analysis: Let $k$ be an algebraically closed field of characteristic $p>0$ (some considerations apply more generally, but we focus on the case of interest to us). Note that if $n=mp^r$ with $(m, p) = 1$ then $\Ext^1(Y, E)[n] = \Ext^1(Y, E)[m]\oplus \Ext^1(Y, E)[p^r]$ and it thus suffices to consider separate cases.
\begin{enumerate}
\item {\bf Case I: $(n, p) = 1$.} \;\; In this case $\Hom(Y[n], E) = \Hom(Y[n], E[n])\cong M_2(\ZZ/n\ZZ)$ is a free $\ZZ/n\ZZ$-module of rank $4$. 
\begin{enumerate}
\item[{\bf Case Ia.}] $E$ and $Y$ are not isogenous. It follows that $\Ext^1(Y, E)[n]\cong (\ZZ/n\ZZ)^4$.
\item[{\bf Case Ib.}] $E$ and $Y$ are isogenous and ordinary. In this case $\Hom(Y, E)$ is either a free $\ZZ$-lattice of rank~$1$ or of rank $2$ (the first case may happen only if $k \not \cong \bar \FF_p$). Consequently, $\Ext^1(Y, E)[n]\cong (\ZZ/n\ZZ)^r$, where $r = 2, 3$, and both cases may happen. 
\item[{\bf Case Ic.}] $E$ and $Y$ are isogenous and supersingular. In this case both can be defined over $\FF_{p^2}$ and $\Hom(Y, E)$ is a rank $4$ $\ZZ$-lattice. It follows that $\Ext^1(Y, E)[n]= 0$.
\end{enumerate}
\item {\bf Case II: $n = p^r$.} \;\; In this case $\Hom(Y[p^r], E) = \Hom(Y[p^r], E[p^r])$ depends on the nature of $E$ and $Y$. 
\begin{enumerate}
\item[{\bf Case IIa.}] $E$ and $Y$ are not isogenous and are both ordinary. Then $\Hom(Y[p^r], E[p^r])\cong \End(\mu_{p^r})\oplus \End(\ZZ/p^r\ZZ) \cong (\ZZ/p^r\ZZ)^2$ and thus $\Ext^1(Y, E)[p^r]\cong (\ZZ/p^r\ZZ)^2$.
\item[{\bf Case IIb.}] $E$ and $Y$ are not isogenous; one is ordinary and the other supersingular. In this case, as for one the $p^r$-torsion is a local-local group scheme and for the other it is a direct sum of an \'etale-local and a local-\'etale group scheme, $\Hom(Y[p^r], E[p^r])=0$ and hence $\Ext^1(Y, E)[p^r] = 0$.

\item[{\bf Case IIc.}] $E$ and $Y$ are isogenous and ordinary. In this case $\Hom(Y, E)$ is either a free $\ZZ$-lattice of rank $1$ or of rank $2$ (the first case may happen only if $k \not \cong \bar \FF_p$). The computation of $\Hom(Y[p^r], E[p^r])$ is as in case IIa. Consequently, $\Ext^1(Y, E)[p^r]\cong (\ZZ/p^r\ZZ)^d$, where $d = 0, 1$, and both cases may happen. 
\item[{\bf Case IId.}] $E$ and $Y$ are isogenous and supersingular. In this case both can be defined over $\FF_{p^2}$, but we consider them over an algebraically closed field $k$; $\Hom(Y, E)$ is a $\ZZ$-lattice of rank $4$. It is well known that $\End(E[p^\infty]) \cong \calO_p$, where $\calO_p$ is the maximal order of the unique division quaternion algebra over $\QQ_p$. It thus follows that 
\[ \Hom(Y, E)/p^r\Hom(Y, E) \cong \calO_p/p^r\calO_p\injects \Hom(Y[p^r], E[p^r]) \cong \End(E[p^r]).\]
But this injection has a non-trivial cokernel. One may calculate $\End(E[p^r])$ using Dieudonn\'e modules. Let $W_r(k)$ denote the Witt vectors of length $r$ over $k$.  Then:
\begin{multline*}
\qquad \qquad \qquad \End(E[p^r])\cong \\ \left\{
\begin{pmatrix} \alpha & p\beta^{\sigma} \\ \beta & \alpha^{\sigma} \end{pmatrix}
\;:\;
\alpha \in W_r(\FF_{p^2}),\;\;
\beta \in W_r(k),\;\;
\beta \bmod p^{r-1} \in W_{r-1}(\FF_{p^2})
\right\},
\end{multline*}
while 
\[\calO_p/p^r\calO_p\cong \left\{
\begin{pmatrix} \alpha & p\beta^{\sigma} \\ \beta & \alpha^{\sigma} \end{pmatrix}
\;:\;
\alpha, \beta \in W_r(\FF_{p^2})
\right\}.\]
It follows that $\Ext^1(Y, E)[p^r]\not\cong 0$ and in fact has cardinality equal to the cardinality of $k$.
\medskip
There are two interesting points to mention in this context; see \cite{Joelle}.
\begin{enumerate}
\item If $0 \arr E \arr A \arr Y \arr 0$ is such that $E, Y$ are supersingular and $a(A) = 1$, then the extension is non-split. For any $E, Y$ such extensions exist and can be obtained as a quotient $E^2/H$, where $H$ is a finite group scheme isomorphic to $\alpha_p$. Every such extension has order $p$.
\item In the same setting, if $a(A) = 2$ then the extension is necessarily split. 
\end{enumerate}
\end{enumerate}
\end{enumerate}

\

\subsubsection{Necessary and sufficient conditions for $p$-foliations} We keep the notation we have used so far, in particular $E, A, Y, \tau, \omega_Y$, and $n$, which denotes the order of the extension class $0 \arr E \arr A \arr Y \arr 0$ in $\Ext^1(Y, E)$. For an elliptic curve $X$ over a field of characteristic $p$, we let 
\[ h(X) = \begin{cases} 0 & $X$ \text{ supersingular,}\\ 1 & $X$ \text{ ordinary.}
\end{cases}\]
\begin{lem}\label{lem:normal form} Let $\tilde \omega_Y = \tau^\ast(\omega_Y)$, where $\omega_Y$ is chosen so that $\scrC_Y(\omega_Y) = h(Y) \cdot \omega_Y$. There exists a holomorphic differential $\tilde \omega_E$ such that $\{\tilde \omega_Y, \tilde \omega_E\} $ is a basis for $H^0(A, \Omega^1_A)$ and relative to it the Cartier operator has a normal form as in the following table, which covers all cases.

\

\begin{center}
\begin{tabular}{l|l|l|l|}
$n$ & $E, Y$& name & $\scrC_A$\\ \hline
$(n, p) = 1$ & any & $\text{\rm Case I}$ & $\left( \begin{smallmatrix} h(Y) & 0 \\ 0 & h(E) \end{smallmatrix} \right)$ \\ \hline
$p\vert n $ & both ordinary & $\text{\rm Case II}_\text{\rm ord}$ & $\left( \begin{smallmatrix} h(Y) & 0 \\ 0 & h(E) \end{smallmatrix} \right)$\\ \hline
$n=p $ & both supersingular & $\text{\rm Case II}_\text{\rm ssg}$ & $\small{\left( \begin{matrix} 0 \;\; & 1 \\ 0\;\;  &0 \end{matrix} \right)}$
\end{tabular}
\end{center}
\end{lem}
\begin{proof}
 Assume that $(n, p) = 1$. Then the isogeny of Lemma~\ref{Lemma:product up to isogeny} induces an isomorphism $\gamma^\ast \colon \pi_E^\ast\Omega^1_E \oplus \pi_Y^\ast\Omega^1_Y \overset{\cong}{\Arr} \Omega^1_A$. We define $\tilde{\omega}_E = \gamma^\ast \omega_E$, where $\omega_E$ is a nonzero holomorphic differential on $E$ such that $\scrC_E(\omega_E) = h(E) \cdot \omega_E$. Then, using the functoriality of the Cartier operator given in Lemma~\ref{functoriality of Cartier}, we find
\[ \scrC_A(\eta) = \tau^\ast (\scrC_Y(f\omega_Y)) + \gamma^\ast\scrC_E(\omega_E), \]
which establishes Case I. 

Consider now Case ${\rm II}_{\rm ord}$. Using Lemma~\ref{functoriality of Cartier}, we find $\scrC_A(\tilde \omega_Y) = \tau^\ast(\scrC_Y(\omega_Y)) = \tilde \omega_Y$. Let $\alpha$ be a holomorphic differential on $A$ such that $\{ \tilde \omega_Y, \alpha\}$ is a basis for $H^0(A, \Omega^1_A)$. As $\scrC_A$ is invertible on $H^0(A, \Omega^1_A)$, in this basis $\scrC_A(\alpha) = a \tilde \omega_Y + b\alpha$, for some $a, b \in k, b\neq 0$. 
For $c, d \in k, d \neq 0$, $\{ \tilde \omega_Y, c\tilde \omega_Y+ d\alpha\}$ is also a basis, and
\[ \scrC_A( c\tilde \omega_Y+ d\alpha) = (c^{1/p}+ d^{1/p}a)\tilde \omega_Y + d^{1/p}b\alpha. \]
This equals $c\tilde \omega_Y + d \alpha$ if and only if both $d^{p-1} = b^p$ and $c^p - c = da^p$ are satisfied with $d \neq 0$. As~$k$ is algebraically closed, such solutions $c, d$ always exist. Thus, Case ${\rm II}_{\rm ord}$ follows.

In any case where $p\vert n$, which is not Case ${\rm II}_{\rm ord}$, the results of \S \ref{sec:extension} imply that both $E$ and $Y$ are supersingular and, moreover, $n=p$. Thus, it only remains to consider Case ${\rm II}_{\rm ssg}$. 

As the extension has order $p$, $A$ is supersingular and not superspecial \cite{Joelle}. This implies that on $H^0(A, \Omega^1_A)$, the Cartier operator satisfies $\scrC_A^2 = 0$, but $\scrC_A \neq 0$. On the other hand, as~$Y$ is supersingular, $\scrC_A(\tilde \omega_Y) = \tau^\ast(\scrC_Y(\omega_Y)) = 0$. Taking any holomorphic differential $\tilde \omega_E$ on $A$ such that $\{ \tilde \omega_Y, \tilde \omega_E\}$ is a basis for $H^0(A, \Omega^1_A)$, one must have $\scrC_A(\tilde \omega_E) = a \tilde \omega_Y$ for some $a\in k^\times$. By rescaling $\tilde \omega_E$, we may assume $a = 1$, and Case ${\rm II}_{\rm ssg}$ follows. 
\end{proof}

\

\begin{prop}\label{cor: when are foliations p closed} Choose a basis $\{\tilde \omega_Y, \tilde \omega_E\}$ such that $\scrC_A$ is in normal form  as in Lemma~\ref{lem:normal form}.  A turbulent foliation defined by $\eta = \tau^\ast(f\omega_Y) + \tilde \omega_E$ is $p$-closed if and only if one of the following holds:
\begin{enumerate} 
\item Case I, $E$ ordinary and $\scrC_Y(f\omega_Y) = f\omega_Y$; equivalently, \cite[Proposition 7]{Cartier}, $f\omega_Y$ is a logarithmic differential, i.e. of the form $\frac{dg}{g}$ for a function $g \in k(Y)^\times$.

\item Case I, $E$ supersingular and $\scrC_Y(f\omega_Y) = 0$; equivalently, \cite[Proposition 8]{Cartier}, $f\omega_Y$ is an exact differential, i.e. of the sort $dg$ for a function $g \in k(Y)$.

\item Case $\text{II}_{ord}$, $E$ ordinary and $\scrC_Y(f\omega_Y) = f\omega_Y$.

\item Case $\text{II}_{ssg}$ and $\scrC_Y(f\omega_Y) + \omega_Y =0$. In this case, one may write $\omega_Y = dh$ and then $f$ is a solution if and only if $f = -h^{p-1} + g_{p-2}^ph^{p-2}+ \cdots + g_0^p$ for any meromorphic functions $g_i \in k(Y)$.
\end{enumerate}
\end{prop}
\begin{proof} Proposition~\ref{prop:p closed in terms of C} showed that $\scrF$ is $p$-closed if and only if $\scrC_A(\eta) = \lambda \eta$ for some $\lambda \in k(A)$. Using Lemma~\ref{functoriality of Cartier} and Lemma~\ref{lem:normal form}, we calculate that 
\[ \scrC_A(\eta) = \begin{cases}  \tau^\ast(\scrC_Y(f\omega_Y)) + h(E) \cdot \tilde \omega_E, & \text{Case I and Case $\text{II}_{\text{ord}}$};\\
\tau^\ast(\scrC_Y(f\omega_Y)) + \tilde \omega_Y, & \text{Case $\text{II}_{\text{ssg}}$.}
    \end{cases}
\]
Therefore, $\scrC_A(\eta) = \lambda \eta$ implies $\lambda = h(E)$. We now consider each case in turn. 
\begin{itemize}
    \item Case I, $E$ ordinary and Case $\text{II}_{\rm ord}$. The condition is then $\scrC_Y(f\omega_Y) = f\omega_Y$.
    \item Case I, $E$ supersingular. The condition is then $\scrC_Y(f\omega_Y) = 0$.
    \item Case $\text{II}_{\rm ssg}$, where $\lambda = 0$. Since $\scrC_Y(\omega_Y) = 0$, $\omega_Y = dh$ for some $h\in k(Y)$. The condition then is $\scrC_Y(f\omega_Y) = -\omega_Y$. Referring to Example~\ref{exa: cartier on curve}, we see $g_{p-1} = -1$ and that is the only condition. 
\end{itemize}
\end{proof}

\

\subsection{Construction of $p$-closed foliations}\label{123}
Based on Proposition~\ref{cor: when are foliations p closed} we therefore change our point of view. In cases (1)--(3) of loc. cit., we start from a logarithmic or an exact differential, as the case may be, define~$f$ accordingly and get turbulent smooth $p$-closed foliations. 
In case (4), we proceed differently and our results are not as complete. 

\subsubsection{Case I and $E$ supersingular} Start then with a non-scalar function $g\in k(Y)$ such that $f = dg/\omega_Y$ is non-constant. This then defines a foliation of the type we want by letting
\[ \eta = \tau^\ast(f \omega_Y) + \tilde \omega_E.\]
We get a foliation $\scrF$ then with poles determined by the effective divisor $(dg/\omega_Y)_\infty =: \sum_{P\in Y} n_P[P]$ on~$Y$. Therefore, given an effective divisor $\delta = \sum_{P\in Y} n_P[P]$ on $Y$, we want to determine if it is of the form $(dg/\omega_Y)_\infty$. 

\begin{thm}\label{p closed thm 1} Assume that $E$ is supersingular and we are in Case I (Proposition~\ref{cor: when are foliations p closed} (2)). Let $\delta$ be an effective divisor on $Y$ of degree $d\geq 2$. If the divisor $\delta = \sum_{P\in Y} n_P[P] = (dg/\omega_Y)_\infty$ for some function $g\in k(Y)$, then for every point~$P$, $n_P \not\equiv 1 \pmod{p}$. 
The converse holds when $\delta$ is not of the form $2[P]$, and also holds for $\delta = 2[P]$ if (and only if) $Y$ is ordinary.

Therefore, if $\delta$ is of the form $ \sum_{P\in Y} n_P[P] $ and, for some $P$, $n_P \equiv 1 \pmod{p}$, then any corresponding foliation defined by a point $(\delta, f) \in \WW^\circ_{Y, d}$ is \emph{not} $p$-closed. If, on the other hand, $\delta$ is of the form $ \sum_{P\in Y} n_P[P] $ and, for all~$P$, $n_P \not\equiv 1 \pmod{p}$, then there exists a $p$-closed foliation defined by some point $(\delta, f) \in \WW^\circ_{Y, d}$,
unless $\delta = 2[P]$ and~$Y$ is supersingular.
\end{thm}
\begin{proof} Given a point $P$, let $x$ be a local uniformizer at $P$. We can expand $g$ and $\omega_Y$ in the form 
\[ g = a_nx^n + a_{n+1}x^{n+1} + a_{n+2}x^{n+2} + \dots, \quad \omega_Y = u(x) dx, \]
for a uniformizer $x$ and a unit $u(x) = \alpha^{-1} + h.o.t.$ of the local ring of $P$, and $a_n \neq 0$. Since $dg/\omega_Y$ for $n\geq 0$ is holomorphic at $P$, we may assume $n<0$ and then
\[ dg/\omega_Y = a_n n \alpha x^{n-1}  + h.o.t.\]
has valuation at $x$ equal to $n-1$ if and only if $p \nmid n$. More precisely, the valuation will be equal to $n+i-1$ for the first $i$ such that $p\nmid  n+i$ and $a_{n+i} \neq 0$. Thus, $n_P -1 = -(n+i-1) -1 = -(n+i)\not\equiv 0 \pmod{p}$.

Conversely, given such a divisor $\delta$, consider the divisor $\sum_{P}(n_P - 1)[P] $. By Riemann-Roch, excluding the case $\delta = 2[P]$, we can find a function $g$ such that $(g)_\infty =\sum_{P}(n_P - 1)[P]$ (it is a function in $\calO_Y(\sum_{P}(n_P - 1)[P] )$ and not in $\calO_Y(\sum_{P}(n_P - 1)[P] - [Q])$ for any $Q$ in the support of $\sum_{P}(n_P - 1)[P]$). Then $\delta = (dg/\omega_Y)_\infty$.  

Finally, for $\delta = 2[P]$ we proceed differently. We can rephrase the question as follows: when is there an exact differential (to wit, $dg$) with polar divisor $2[P]$? Such a differential is a global section of $H^0(Y, \Omega^1_Y(2[P]))$, which is a $2$-dimensional space. Let $x_P$ be a global  function on $Y$ with pole of order $2$ at $P$ and otherwise regular. Then, $H^0(Y, \Omega^1_Y(2[P]))$ has a basis $\omega_Y, x_P\omega_Y$ and so we ask: when is a differential of the form $\nu = \alpha\omega_Y + \beta x_P\omega_Y$, with $\alpha, \beta \in k, \beta \neq 0$, an exact differential? By Cartier, this is the case if and only if $\scrC_Y(\nu) = 0$. 

Remark that a uniformizer $t$ at a point $Q$ on $Y$ is a $p$-basis and so we can use the expression for the Cartier operator in \S \ref{subsec:Cartier} also as a local expression at $Q$. At every point $Q\neq P$ of $Y$ the differential $x_P \omega_Y$ is regular, and so $\scrC_Y(x_P\omega_Y)$ is regular at $Q$. On the other hand, letting $t$ be a local uniformizer at $P$ we have $x_P\omega_Y = (a_{-2}t^{-2} + a_0 + h.o.t.)dt$ and $\scrC_Y(t^{-2}dt) = \scrC_Y(t^{p-2}t^{-p}dt) = t^{-1}\scrC_Y(t^{p-2}dt) = 0$. Thus, $\scrC_Y(x_P\omega_Y)$ is regular at $P$ too and so there exists $\gamma \in k$ such that $\scrC_Y( x_P \omega_Y) = \gamma \omega_Y$. We now find
\[ \scrC_Y(\nu) = \alpha^{1/p} \scrC_Y(\omega_Y) + \gamma \beta^{1/p} \omega_Y. \]

If $Y$ is ordinary, we can choose $\omega_Y$ such that $\scrC_Y(\omega_Y) = \omega_Y$ and solve the equation by taking $\beta = 1$ and $\alpha^{1/p} = -\gamma$. Thus, if $Y$ is ordinary and $\delta = 2[P]$ there exists a $p$-closed foliation defined by some point $(\delta, f) \in \WW^\circ_{Y, d}$.

On the other hand, if $Y$ is supersingular, $\scrC_Y(\omega_Y) = 0$ and so we can find $\nu$ such that $\scrC_Y(\nu) = 0$ if and only if $\gamma = 0$. We claim that this is not possible, namely that \[\scrC_Y(x_P\omega_Y) \neq 0.\]

To simplify, we write $\omega = \omega_Y, \scrC = \scrC_Y$. The cases of $p=2, 3$ are easy to verify directly.
We perform the computation using Manin's method \cite{Manin}. We use the Legendre equation $y^2 = f(x) = x(x-1)(x-\lambda)$ ($p>2$) and  find that 
\[C(\omega) = A^{1/p} \omega, \qquad C(x \omega) = B^{1/p} \omega,\]
where $A$ is the coefficient of $x^{p-1}$ in $f(x)^m$ and $B$ is the coefficient of $x^{p-2}$ in that polynomial. Again, $A$ is proportional to the Hasse invariant. Both $A$ and $B$ have elementary expressions:
\[A = (-1)^m \sum_{b=0}^{m} \binom{m}{b}^{2} \lambda^{b},
\qquad
B = (-1)^{m+1} \sum_{b=1}^{m} \binom{m}{b-1}\binom{m}{b} \lambda^{b},
\qquad m = \tfrac{p-1}{2}.\]
They satisfy the following relation
\[B = \lambda(A + 2(\lambda-1)A'(\lambda)) \pmod{p},\] 
where $A'$ is the derivative of $A$ relative to the variable $\lambda$. There are more and less sophisticated methods to prove this relation. The simplest we found, much simpler than Claude's original suggestion, is to use that the relation is equivalent to the identities
\[-\binom{m}{b-1}\binom{m}{b} \equiv (2b-1)\binom{m}{b-1}^{2} - 2b\binom{m}{b}^{2} \pmod{p},\]
which one might prove by patiently expanding both sides (the first step would be to multiply through by $\frac{(b! \cdot (m-b+1)!)^{2}}{(m!)^2}$).

Since $Y$ is supersingular, $A = 0$ and the relation simplifies to $B =  2\lambda(\lambda-1)A'(\lambda) \pmod{p}$. By a result of Igusa, the supersingular polynomial $A(\lambda)$ has simple roots at the supersingular $\lambda$'s and so $B$ is not zero.  
\end{proof}

\begin{exa} Let $d \geq 3$ and $p>2$. Consider $\WW^\circ_{Y, d} \arr \Sym^dY$. For a general point $\delta$ of $\Sym^dY$ the divisor~$\delta$ has multiplicity $1$ at every point $P$ of its support and thus the corresponding foliations are not $p$-closed. To simplify, suppose that $d \leq p$. The locus $Z$ where all the multiplicities are not congruent to~$1$ modulo $p$ is in general reducible, but it has a unique component of maximal dimension, which is equal to $\lfloor d/2\rfloor$ (for $d$ odd, choose one point to be of multiplicity $3$ and the rest of multiplicity $2$; for $d$ even, choose all points to have multiplicity $2$). In this case, $Z$ is closed and the locus of $p$-closed foliations in $\WW^\circ$ surjects onto $Z$ (and it intersects properly and non-trivially each fiber over $Z$). But for $d = p+1$ the locus $Z$ is not closed because $(p+1)[Q]$, for any point $Q$ of $Y$, is not in $Z$ but is a limit point of every component of $Z$.

\end{exa}

\begin{exa}
For another example, take $p=3$, and let $Y_\lambda$ be the elliptic curve
\[Y_\lambda:\quad  y^2 = x(x-1)(x-\lambda), \]
where $\lambda \in \fpbar, \lambda \neq 0, 1$. It is a non-singular elliptic curve. Assume that $Y$ is supersingular. As $(p-1)/2 = 1$, one finds immediately that 
\[ \calC_Y(\omega_Y) = -(1+\lambda)^{1/3} \omega, \quad \calC_Y(x\omega) = \lambda^{1/3}\omega.\]
Incidentally, that shows that $Y_\lambda$ is supersingular if and only if $\lambda = -1$.

Take $d = 4$. By our results, there exist $p$-closed foliations over $(\delta, f)$ where $\delta \in \Sym^4 Y$ if and only if the support of $\delta$ consists of exactly $2$ points. That is, $\delta$ must be of the form $2P + 2Q$ for some points $P \neq Q$ on $Y_\lambda$. The set $Z$ is isomorphic to $(Y^2 - \Delta(Y))/\Sigma_2$, where $\Delta$ is the image of the diagonal embedding  $Y \injects Y^2$. In particular, $Z$ is not closed. 

We want to understand the fibers of the morphism $\WW^\circ_{Y, 4} \arr \Sym^4 Y$ over a point $\delta\in Z$. Writing $\delta = 2[P]+2[Q]$, the fiber over $\delta$ of $\WW_{Y, 4} \arr \Sym^4 Y$ is $H^0(Y, 2[P]+2[Q])$. Using translation symmetries, we may assume for convenience that $[P] = 0_{Y_\lambda}$. This is a $4$-dimensional vector space that has a basis of the form $\{1, x, g, h\}$ where $g, h$ are functions with the property that $g$ has polar divisor $2[P]+[Q]$ and $h$ has a polar divisor $2[P]+2[Q]$. We further normalize $x$ and~$h$ so that their expansions in terms of a local parameter $t$ at $[0_Y]$ begin with $t^{-2}$. For $b, d \in k$, consider then $f = bx + dh\in H^0(Y, 2[P]+2[Q])$. It has polar divisor $2[P]+ 2[Q]$ as long as $d(b+ d) \neq 0$. On the other hand,
\[ \calC_Y(f\omega) = -b^{1/3} \omega + d^{1/3} \calC_Y(h\omega).\]
It follows that whether $\calC_Y(h\omega) = 0$ or not, one can choose $b, d$ such that both $d(b+ d) \neq 0$ and $\calC_Y(f\omega) \neq 0$, thus showing that in the fiber of $\WW^\circ_{Y, 4}$ over $\delta = 2[P]+2[Q]$ (for $\delta \in Z$, in fact!) there are points corresponding to $p$-closed foliations
and there are points corresponding to foliations that are not $p$-closed. 
\end{exa}

\

\subsubsection{Case I and $E$ ordinary, and Case $\text{II}_{\text{ord}}$} In this case we want $f\omega_Y = dg/g$ for some $g \in k(Y)^\times$ and non-constant $f$. Starting with any $g\in k(Y)^\times$ we define 
\[ f =  \frac{dg}{g \omega_Y}.\]
We want to find the possibilities for $(\frac{dg}{g \omega_Y})_\infty$.

As before, for a point $P$ and a uniformizer $x$ at $P$, write
\[ g = a_nx^n + a_{n+1}x^{n+1} + a_{n+2}x^{n+2} + \dots \; (a_n \neq 0), \quad \omega = u(x) dx,\]
but now we consider every possible $n$. For $n=0$, $g$ is a unit at $P$ and $dg$ is regular at $P$, so $P$ does not appear in $(\frac{dg}{g \omega_Y})_\infty$. For $n\neq 0$,
\[  \frac{dg}{g \omega_Y} = u(x)^{-1}\cdot \frac{na_nx^{n-1} + (n+1)a_{n+1}x^{n} + (n+2)a_{n+2}x^{n+1} + \dots}{a_nx^n + a_{n+1}x^{n+1} + a_{n+2}x^{n+2} + \dots}.\]
If $p \vert n$, there is no pole at $P$, while if $p \nmid n$ there is a simple pole at $P$. We conclude that 
\[\left(\frac{dg}{g \omega_Y}\right)_\infty = \sum_{i=1}^r [P_i], \]
where the sum is over distinct points $P_i$. 

\ 

Conversely, let $\delta = \sum_{i=1}^r [P_i]$ be such a divisor. We claim that there is a function $g\in k(Y)^\times$ such that $\left(\frac{dg}{g \omega_Y}\right)_\infty = \delta$ if and only if there are integers $a_i, i=1, \dots, r$, such that $\forall i, (a_i, p) = 1$ and $a:=\sum_{i=1}^r a_i \equiv 0 \pmod{p}$. 

Indeed, if such a function $g$ exists, the computation above shows that for all $P_i$, $a_i :=\val_{P_i}(g)$ is not divisible by $p$, and for any other point $Q$, $\val_Q(g) \equiv 0 \pmod{p}$. Since the divisor of a function has degree~$0$, $a \equiv 0 \pmod{p}$. In the other direction, if such $a_i$ exist, by modifying $a_1$ to $a_1+p$ we may assume $a \neq 0$; choose a point $Q$ such that $[a]Q = \sum_{i=1}^r [a_i] P_i$ (we use $[r]$ to indicate the multiplication by $r$ map on $Y$). The divisor $\sum_{i=1}^r a_i[P_i]-a[Q]$ is then principal, equal to $(g)$ for some function $g$ and $\left(\frac{dg}{g \omega_Y}\right)_\infty = \sum_{i=1}^r [P_i]$. To remove any doubts, this argument works even if $Q$ is one of the $P_i$.

It therefore remains to analyze the condition 
\[ \exists a_1, \dots, a_r, \forall i \;(a_i, p) = 1, \quad \sum_{i=1}^r a_i \equiv 0 \pmod{p}.\]
Assume first that $r>1$ and $p>2$. Then such integers $a_i$ can always be found. If $r \not\equiv 1 \pmod{p}$ take $(a_1, \dots, a_r)=(1^{r-1}, p-(r-1))$ and if $r \equiv 1 \pmod{p}$ take $(a_1, \dots, a_r)=(1^{r-2}, 2, p-1)$.

If $r = 1$, there are no such integers $a_i$. 
Finally, if $r>1$ and $p=2$ then such integers $a_i$ exist if and only if $r$ is even.

\ 

We summarize our discussion in the following theorem. 

\begin{thm}\label{p closed thm 2} Assume that $E$ is ordinary and we are in Case I or Case $\text{II}_{\text{ord}}$. Let $\delta$ be an effective divisor on $Y$ of degree $d\geq 2$. If the divisor $\delta = \sum_{P\in Y} n_P[P] = (dg/g\omega_Y)_\infty$ for some function $g\in k(Y)$, then for every point~$P$, $n_P = 0$ or $1$.

Therefore, if $\delta$ is an effective divisor of the form $ \sum_{P\in Y} n_P[P] $ and, for some $P$, $n_P > 1$, then \emph{any} corresponding foliation defined by a point $(\delta, f) \in \WW^\circ_{Y, d}$ is  \emph{not $p$-closed}. 

If, on the other hand, $\delta$ is an effective divisor of the form $ \sum_{i=1}^r [P_i]$, then \emph{there is} a \emph{$p$-closed} foliation defined by a point $(\uy, f) \in \WW^\circ_{Y, d}$ such that $\uy =\delta$ if and only if there exist integers $a_1, \dots, a_r$, $(a_i, p) = 1$, $\sum_{i=1}^r a_i \equiv 0\pmod{p}$. Such integers exist if and only if $r>1$, and $r$ is even when $p=2$.
\end{thm}

\

\subsubsection{Case $\text{II}_{\text{ssg}}$} In this case our results are not as complete as in previous cases. To fix ideas, suppose that $p>2$ and the supersingular elliptic curve $Y$ is given by a Weierstrass equation
\[ Y\colon \quad y^2 = g(x),\]
where $g(x)$ is a monic separable cubic polynomial. We may take $\omega_Y = \frac{dx}{y}$. Let $Q_1, Q_2, Q_3$ be the nonzero points on $Y[2]$ and define the divisor \[\delta_2 = [Q_1]+[Q_2]+[Q_3].\]
We write
\[ \omega_Y = y^{-p}g(x)^{(p-1)/2} dx.\]
Then $\scrC_Y(\omega_Y) = y^{-1}\scrC_Y(g(x)^{(p-1)/2} dx)$. Now, as $g(x)^{(p-1)/2} = x^{3(p-1)/2} + l.o.t.$, the only obstruction to integrating $g^{(p-1)/2}$ (namely, finding a function $G\in k[x]$ such that $dG = g^{(p-1)/2}dx$) is that the coefficient of $x^{p-1}$ is nonzero; at the same time, referring to Example~\ref{exa: cartier on curve}, the vanishing of this coefficient is equivalent to $\scrC_Y(\omega_Y) = 0$. We conclude that there is a function $G \in k[x]$ such that $dG = g^{(p-1)/2}dx$ and $G = -2 x^{1+ 3(p-1)/2} + l.o.t.$. By replacing $G$ by $G+c$ for a suitable scalar $c$, we may assume that $G$ does not vanish along $\delta_2$.

Let $h = y^{-p} G$. As $(y^{-p}) = 3p [0_Y] - p\delta_2$ and $(G)_\infty = (3p-1)[0_Y]$, we find that
\[ dh = \omega_Y, \qquad (h)_\infty =  p\delta_2.\]

\begin{thm}\label{thm: p closed tricky} Let $p>2$.  Let $\delta\geq 0$ be any effective divisor on $Y$ such that ${\rm supp}(\delta) \cap \delta_2 = \emptyset$. There is then a function $f\in k(Y)$ such that $\scrC_Y(f\omega_Y) +\omega_Y=0$ and \[(f)_\infty = p((p-1)\delta_2 + \delta).\] Therefore, there are $p$-closed smooth turbulent foliations $\scrF$ on $A$ such that 
\[ \scrF \cong \calO_A(-D), \quad D = p\tau^\ast((p-1)\delta_2 + \delta).  \]
\end{thm}
\begin{proof} First, note that the function $h$ is a $p$-basis. Referring to \S~\ref{subsec:Cartier}, in particular Example~\ref{exa: cartier on curve}, we find that the most general form of a function $f\in k(Y)$ such that $\scrC_Y(f\omega_Y) = -\omega_Y$ is a function with a $p$-basis expansion
\[ f = -h^{p-1} + g_{p-2}^ph^{p-2} + \cdots + g_0^p,  \]
where $g_i \in k(Y)$. Let $\delta$ be an effective divisor of degree $d\geq 2$ with support disjoint from $\delta_2$. By Riemann-Roch, there is a function $g_0\in k(Y)$ such that $(g_0)_\infty = \delta$. Take then $f = -h^{p-1} + g_0^p$. If~$\delta$ is a single point, there is a function $g_0$ such that $(g_0)_\infty = Q_1 + \delta$. Take again $f =  -h^{p-1} + g_0^p$. If $\delta = 0$, we simply take $f = -h^{p-1}$.
\end{proof}

\medskip
\begin{rmk} The proof shows that further variations are possible, even restricting to combinations of the form $-h^{p-1} + g_0^p$, and thus the result of Theorem~\ref{thm: p closed tricky} is not the last word on the subject. However, we were not able to characterize all the divisors that may arise as $(f)_\infty$, for $f = -h^{p-1} + g_{p-2}^ph^{p-2} + \cdots + g_0^p$.
\end{rmk}

\

\subsection{Linear $p$-closed foliations}\label{subsec:linear p closed} For completeness, we discuss $p$-closed linear foliations on abelian varieties $A$ of arbitrary dimension $d$ over an algebraically closed field $k$ of characteristic $p$. A rank $r$ linear foliation on $A$ is a translation-invariant rank $r$ saturated subsheaf of $\calT_A$. These are in bijection with $k$ vector subspaces of the tangent space to $A$ at the origin $D = \calT_{A, 0}$. In fact, in characteristic $p$, we may view $D$ as the covariant Dieudonn\'e module of the subgroup scheme $G:=\Ker(\Fr_A\colon A \arr A^{(p)})$. As such, $p$-closed foliations $\scrF$ (of rank $r$) correspond to Dieudonn\'e submodules of $D$, which, in turn, correspond to finite group schemes $H$ of $A$ (of rank $p^r$) contained in $G$. Moreover, the quotient of $A$ by $\scrF$ in the sense of foliations is precisely the quotient $A/H$ in the sense of abelian varieties.  

It is thus quite a simple matter to classify those. Supposing, for example, that $A$ has a polarization of degree prime to $p$, one can use the work of Ekedahl and Oort \cite{EO} to enumerate all the possibilities. 

For illustration, we consider the case of an abelian surface $A$, regardless of polarization. In this case $G$ has rank $p^2$. Using the classification of commutative group schemes over $k$, we have the following possibilities:
\begin{enumerate}
\item $G \cong \alpha_p \oplus \alpha_p$. This happens if and only if $A$ is a superspecial surface. In this case, the linear foliations are in bijection with lines in $D$, namely with $\PP^1(k)$. Each linear foliation corresponds to a copy of the group scheme $\alpha_p$ inside $A$. In this case $A/H$ is a supersingular abelian surface that could be either superspecial, or not. 
\item $G$ is a non-split extension of $\alpha_p$ by $\alpha_p$. This happens if and only if $A$ is a supersingular, but not superspecial, abelian surface. Thus, there is a unique $p$-closed foliation, corresponding to the unique copy of $\alpha_p$ inside $A$. In this case $A/H$ is a superspecial abelian surface.
\item $G \cong \mu_p \oplus \alpha_p$. This case can arise in two different ways: (i) $A$ is isogenous to a product of a supersingular elliptic curve by an ordinary elliptic curve, or (ii) $A$ is a simple abelian surface. In fact, the second case is the generic case. If one considers such $A$ that are principally polarized, there is a $2$-dimensional irreducible surface in the moduli space that parametrizes such simple surfaces, while the moduli of case (i) is a countable union of curves. In either case, there are precisely two $p$-closed linear foliations; they correspond to $\mu_p$ and $\alpha_p$ inside $A$.
\item $G \cong \mu_p \oplus \mu_p$. In this case $A$ is ordinary and the rank-$1$ $p$-closed linear foliations correspond to subgroups of order $p$ of $\mu_p^2$. There are $p+1$ such subgroups.  
\end{enumerate}

\

\

\section{Families of foliations}
Recall our conventions: $k$ is an algebraically closed field, $Y$ is an elliptic curve over $k$ and $d\geq 2$ is an integer. We recall that given a scheme $S$ over $k$, an effective Cartier divisor $\delta\subset Y\times_k S$ of degree $d$, relative to $S$, is a closed subscheme $\delta$ of $Y\times_k S$,  finite and flat of degree $d$ over $S$ and such that the ideal sheaf $\calO_{Y\times_k S}(-\delta)\subset \calO_{Y\times_k S} $ defining $\delta$ is an invertible $\calO_{Y\times_k S}$-module (cf. \cite[Definition 3.4]{Milne}). 

Given a $k$-scheme $S$ we denote by $\mathrm{Div}^d_Y(S)$ the set of effective Cartier divisors $\delta_S\subset Y\times_k S$, relative to $S$ of degree~$d$. Given a morphism of $k$-schemes $f\colon T \to S$, one has a natural notion of pullback of relative Cartier divisors $f^\ast\colon\mathrm{Div}^d_Y(S) \to \mathrm{Div}^d_Y(T) $. We get in this way a contravariant functor $\mathrm{Div}^d_Y$ from $k$-schemes to sets. 

We define $\Sym^d Y$ to be the quotient $Y^d/\Sigma_d$ of the $d$-fold product of $Y$ over $k$, modulo the action of the group of permutations of~$d$ elements. We then have the following

\begin{prop}\label{prop: Symd} The quotient $\Sym^d Y$ is a projective and smooth variety over $k$, of dimension~$d$. Moreover, it represents the functor $\mathrm{Div}^d_Y$.
\end{prop}

We refer to \cite[Proposition 3.2 \& Theorem 3.13]{Milne} for a proof.

\subsection{$\WW^\circ_{Y, d}$ as a moduli space}\label{sec: WW}
Let $\pi\colon Y\times_k \Sym^d Y  \to \Sym^d Y $ be the projection map. Let $\delta\subset Y\times_k \Sym^d Y $ be the corresponding  universal relative Cartier divisor. Let $\calW:=\pi_\ast\bigl( \calO_{Y\times_k \Sym^d Y }(\delta)\bigr)$. Since $\pi$ is projective, $\calW$ is a coherent $\calO_{\Sym^d Y }$-module.

\begin{lem} The $\calO_{\Sym^d Y }$-module $\calW$ is locally free of rank $d$, and for any point~$t$ of $\Sym^d Y$, the natural map $\calW\otimes_{\calO_{\Sym^d Y}} k(t) \arr H^0(Y, \calO_Y(\delta_t))$ is an isomorphism. 
\end{lem}
\begin{proof} Since $Y$ is an elliptic curve, it follows from the Riemann-Roch theorem that for every point~$t$ of $\Sym^d Y $ we have $h^0(Y,\calO_Y(\delta_t))=d$. Since $\Sym^dY$ is reduced and connected, and $\calO_Y(\delta_t) \cong 
(\calO_{Y\times_k \Sym^d Y }(\delta))_t$, the claim follows, for example, from \cite[\S II.5, Corollary 2]{MumfordAbelianVarieties}.
\end{proof}

We denote by $\WW_{Y, d}$ the vector bundle over  $\Sym^d Y $ associated to the $\calO_{\Sym^d Y }$-module $\calW$. Given a $k$-algebra $R$, an $R$-valued point of $\WW_{Y, d}$ corresponds to a pair $(\delta_R,f_R)$ where $\delta_R$ is an effective Cartier divisor $\delta_R\subset Y_R:=Y\times_k \Spec(R)$ of degree $d$, relative to $\Spec(R)$, and  $f_R$ is a section $f_R\in H^0\bigl(Y_R, \calO_{Y_R}(\delta_R)\bigr)$. In particular, over $\WW_{Y, d}$ we have a universal Cartier divisor $\delta_{\WW_{Y, d}} \subset Y\times_k \WW_{Y, d}$ and a universal section $f_{\WW_{Y, d}}\in H^0\bigl(Y\times_k \WW_{Y, d}, \calO_{Y\times_k \WW_{Y, d}}(\delta_{\WW_{Y, d}})\bigr)$.  We claim:

\begin{lem}\label{lem quotientofO2} There is a non-empty open subvariety $\WW_{Y, d}^\circ\subset \WW_{Y, d}$ whose $R$-valued points, for a $k$-algebra $R$, are characterized by the fact that the sections $1$ and the base change $f_R$ of $f_{\WW_{Y, d}}$ together span $\calO_{Y_R}(\delta_R)$ as an $\calO_{Y_R}$-module.
\end{lem}
\begin{proof} The section $f_{\WW_{Y, d}}$ defines a morphism $F_{\WW_{Y, d}}\colon \calO_{Y\times_k \WW_{Y, d}} \arr \calN_\delta$, of sheaves on $Y\times_k \WW_{Y, d}$, as the composition
\[\xymatrix@C=1.5cm{\calO_{Y\times_k \WW_{Y, d}} \ar[r]^-{\times f_{\WW_{Y, d}}} & 
\calO_{Y\times_k \WW_{Y, d}}(\delta_{\WW_{Y, d}}) \ar[r] & \calO_{Y\times_k \WW_{Y, d}}(\delta_{\WW_{Y, d}})/\calO_{Y\times_k \WW_{Y, d}}\cdot 1\cong \calN_\delta } \]
The quotient of $\calO_{Y\times_k \WW_{Y, d}}(\delta_{\WW_{Y, d}})$ by the $\calO_{Y\times_k \WW_{Y, d}}$-module spanned by the constant function $1$ is isomorphic to the normal sheaf $\calN_\delta$ for the inclusion $\delta_{\WW_{Y, d}}\subset Y\times_k \WW_{Y, d}$, i.e.,  the invertible $\calO_{\delta_{\WW_{Y, d}}}$-module  given by the pullback of $\calO_{Y\times_k \WW_{Y, d}}(\delta_{\WW_{Y, d}})$ to $\delta_{\WW_{Y, d}}$. The image of $F_{\WW_{Y, d}}$ defines an $\calO_{\delta_{\WW_{Y, d}}}$-submodule $\calF_{\WW_{Y, d}}$ of $\calN_\delta$. Since $\delta_{\WW_{Y, d}}$ is finite and locally free over $\WW_{Y, d}$, the locus in $\WW_{Y, d}$ where the inclusion $\calF_{\WW_{Y, d}}\subset \calN_\delta$ is an isomorphism is open in $\WW_{Y, d}$. This coincides with $\WW_{Y, d}^\circ$. 

It remains to prove that $\WW_{Y, d}^\circ(k)$ is non-empty. This follows as, given a divisor $\delta\subset Y$ of degree $d$, the Riemann-Roch theorem implies that there exists an $f \in H^0\bigl(Y,\calO_Y(\delta)\bigr)$ with polar divisor exactly $\delta$, and so $1$ and $f$ span $\calO_Y(\delta)$ as an $\calO_Y$-module.
\end{proof}

In particular, the sections  $1$ and $f_{\WW_{Y, d}^\circ}:=f_{\WW_{Y, d}}\vert_{\WW_{Y, d}^\circ}$ of $\calO_{Y\times_k \WW_{Y, d}}(\delta_{\WW_{Y, d}})$ over $\WW_{Y, d}^\circ$ span the sheaf $\calO_{Y\times_k \WW_{Y, d}^\circ}(\delta_{\WW_{Y, d}})$ so we get a morphism  $$\Psi_{\WW_{Y, d}^\circ}\colon Y\times_k \WW_{Y, d}^\circ \longrightarrow \PP^1_k\times_k \WW_{Y, d}^\circ,$$over $\WW_{Y, d}^\circ$ such that the pullback of $\calO_{\PP^1_k\times_k \WW_{Y, d}^\circ}(1)$ is $\calO_{Y\times_k \WW_{Y, d}^\circ}(\delta_{\WW^\circ_{Y, d}})$ and the pullbacks of the two standard sections of $\calO_{\PP^1_k}(1)$ are the sections $1$ and $f_{\WW^\circ_{Y, d}}$ respectively. 

\begin{prop}\label{prop PsiR} The morphism $\Psi_{\WW_{Y, d}^\circ}$ is finite and flat of degree $d$ and the pullback of the relative Cartier divisor~$\infty$ of $\PP^1_k\times_k \WW_{Y, d}^\circ$ is $\delta_{\WW_{Y, d}^\circ}$.  

Vice versa, for every $k$-algebra $R$, the $R$-valued points of $ \WW_{Y, d}^\circ$ are in one-to-one correspondence with finite and flat morphisms $\Psi_R\colon Y_R \to \PP^1_R$, of degree $d$.

\end{prop}
\begin{proof} For the first part of the proposition, it remains to show that $\Psi_{\WW_{Y, d}^\circ}$ is finite and flat of degree $d$. 
The morphism $\Psi_{\WW_{Y, d}^\circ}$ is quasi-finite. It is also proper as it is a morphism of schemes projective over $\WW_{Y, d}^\circ$. It follows that it is finite. It is flat pointwise on $\WW_{Y, d}^\circ$ and, hence, it is flat by the fiberwise criterion of flatness \cite[Lemma 37.16.3]{StackProject}. The pullback via an $R$-valued point of $\WW_{Y, d}^\circ$ then defines a finite and flat morphism $\Psi_R\colon Y_R \to \PP^1_R$, of degree $d$. 

Vice versa, given a finite and flat morphism $\Psi_R\colon Y_R \to \PP^1_R$ of degree $d$, the pullback of the Cartier divisor $\infty$ defines a relative Cartier divisor $\delta_R\subset Y_R$ of degree $d$ over $\Spec(R)$. We also get two sections $1$ and $f_R$ corresponding to the two standard sections of $\calO_{\PP^1_k}(1)$. Hence, $(\delta_R,f_R)$ defines an $R$-valued point of $\WW_{Y, d}$. Since $1$ and $f_R$ span $\calO_{Y_R}(\delta_R)$, by definition this $R$-valued point factors through $\WW_{Y, d}^\circ$.    
\end{proof}

\subsection{The Quot scheme ${\rm Quot}^P$ and the morphism $\WW_{Y, d}^\circ \arr {\rm Quot}^P$}\label{subsec: morphism to quot}

Let $A$ be an abelian surface. Our notation is as in \S \ref{sec Hilbertpol}. Consider the polynomial $P(t)  = n^6t^2 + n^4 d t$  introduced in (\ref{eqn:P}). Denote by ${\rm Quot}^P$ the scheme classifying quotient $\calO_A$-modules $\calO_A^2\to \calN$ with Hilbert polynomial  equal to $P(t)$; see \cite[Theorem 2.2.4]{HL}.

It follows from Lemma \ref{lem quotientofO2} that over $Y\times_k \WW_{Y, d}^\circ$ we have a surjective map $\calO_{Y\times_k \WW_{Y, d}^\circ}^2\to \calO_{Y\times_k \WW_{Y, d}^\circ}(\delta_{\WW_{Y, d}^\circ})$ given by the two sections $1$ and $f_{\WW_{Y, d}^\circ}$.  Let $D_{\WW_{Y, d}^\circ}\subset A\times_k \WW_{Y, d}^\circ$ be the pullback of the relative Cartier divisor $\delta_{\WW_{Y, d}^\circ}$ via $\tau$. Pulling back via $\tau$ we get over $A\times_k \WW_{Y, d}^\circ$ a surjective morphism  $\upsilon\colon \calO_{A\times_k \WW_{Y, d}^\circ}^2\to \calO_{A\times_k \WW_{Y, d}^\circ}(D_{\WW_{Y, d}^\circ})$ of $\calO_{A\times_k \WW_{Y, d}^\circ}$-modules. It follows from \S \ref{sec Hilbertpol} that $\calO_{A\times_k \WW_{Y, d}^\circ}(D_{\WW_{Y, d}^\circ})$ has Hilbert polynomial $P(t)$ and so, by the universal property of ${\rm Quot}^P$, that $\upsilon$ defines, and is defined by, a morphism of schemes over $k$: $$\rho\colon \WW_{Y, d}^\circ\longrightarrow {\rm Quot}^P.$$Over $A\times_k {\rm Quot}^P $ we have a universal quotient $\calO_{A\times_k {\rm Quot}^P}^2 \to \calN^{\rm univ}$. We let \[{\rm Quot}^{P, \circ}\subset {\rm Quot}^P\] be the locus where $\calN^{\rm univ}$ is an invertible $\calO_{A\times_k {\rm Quot}^P}$-module. Then, ${\rm Quot}^{P, \circ}\subset {\rm Quot}^P$ is open by \cite[Lemma 2.1.8]{HL} and $\rho$ factors via ${\rm Quot}^{P, \circ}$.

\begin{prop}\label{prop Quotosmooth} The tangent space of ${\rm Quot}^{P, \circ}$  at $k$-points  in the image of $\rho$ is of dimension $2d$.
\end{prop}
\begin{proof}
Let $x$ be a $k$-valued point of ${\rm Quot}^{P,\circ}$ corresponding to a quotient $\calO_A^2\to \calN$ with $\calN$ an invertible sheaf. Let $\scrF \subset \calO_A^2$ be the kernel of the projection to $\calN$. It follows from \cite[Proposition 2.2.7]{HL} that the tangent space of ${\rm Quot}^{P,\circ}$ at~$x$ is identified with $\Hom_{\calO_A}(\scrF,\calN)$. Since $\calN$ is invertible,  we find $\Hom_{\calO_A}(\scrF, \calN)\cong H^0(A,\scrG)$ with $\scrG=\calN\otimes_{\calO_A} \scrF^\vee\cong \calN^2$. 

Assume that  $\calN = \tau^\ast \calO_Y(\delta)$. Using that the morphism $A \arr Y$ is proper with connected fibers, and making use of the projection formula, we get $\tau_\ast(\scrG)= \tau_\ast\!\left(\tau^\ast\calO_Y(2\delta)\right) \;\cong\; \left( \tau_\ast\,\mathcal{O}_A\right) \otimes_{\calO_Y}\calO_Y(2\delta)$. Since the morphism $\tau$ is proper with connected fibers, the latter is $\calO_Y(2\delta)$. We find that $H^0(A, \scrG) \cong H^0(Y, \calO_Y(2\delta))$ and the latter is $2d$-dimensional by Riemann-Roch.
\end{proof}

\ 

\begin{thm} \label{W and quot} The morphism $\rho$ identifies $\WW^\circ_{Y, d}$ with an irreducible component of ${\rm Quot}^{P, \circ}$; this component is therefore smooth and irreducible of dimension $2d$, naturally fibered over $\Sym^d Y$ with geometric fibers isomorphic to the complement of finitely many codimension $1$ subspaces in $\AA^d_k$.
\end{thm}
\begin{proof}
Let $S'\subset {\rm Quot}^{P, \circ}$ be the connected component containing the image of $\rho$. Take $S$ to be $S'$  with its reduced induced structure. Let $\calO_{A_S}^2\to \calN$ be the universal quotient with $\calN$ an invertible $\calO_{A_S}$-module  with Hilbert polynomial $P(t)$. We claim that $\calO_{A_S}^2\to \calN$ is obtained by pullback via $\tau\colon A \to Y$ of a surjection $\calO_{Y_S}^2 \arr \calL$, where $\calL$ is an invertible $\calO_{Y_S}$-module of degree~$d$. 

Let $\lambda_\calN\colon A_S \to A_S^\vee$ be the homomorphism of abelian $S$-schemes defined by $\calN$. Since the Hom-scheme $\mathrm{Hom}(A,A^\vee)$ is the disjoint union of copies of $\Spec(k)$, $\lambda_\calN$ is constant on $S$. Computing at one point in the image of $\rho$, we deduce that $\lambda_\calN= \tau^\vee \circ \lambda_Y \circ \tau$ for $\lambda_Y$ a polarization on $Y$ of degree $d^2$.  Let $E\subset A$ be the elliptic curve given by the kernel of $\tau$. 

Consider the Mumford theta group $\scrG^+$ of $\calN$ (see, e.g., \cite{MB}), which sits in an exact sequence
\[1 \arr \GG_{m, S} \arr \scrG^+ \arr \scrK \arr 0,\]
and let $\scrG$ be its restriction to $E_S$. Note that $\lambda_\calN$ vanishes on $E_S$ and therefore $E_S \subseteq \scrK$. Now, $\scrG$ is contained in the Mumford theta group of $\calN\vert_{E_S}$ and thus we conclude that $\scrG$ is an extension of $E_S$ by $\GG_{m,S}$ and coincides with the Mumford theta group of $\calN\vert_{E_S}$. We now prove that $\scrG$ is a split extension. 

First of all, viewing~$\scrG$ as a $\GG_m$-bundle over $E_S$ whose associated line bundle is invariant under translations by $E_S$ (hence algebraically equivalent to $0$), $\scrG$ defines and is defined by a point $g\in E^\vee(S)$. So we have to show that $g$ is the zero section.  Since $S$ is reduced, it suffices to show that this holds true for any $k$-valued point of $S$. Take $s\in S(k)$ and consider the associated surjection $\calO_A^2\to \calN_s$. We know from the discussion in \S \ref{sec turbfol}  that there exists a quotient elliptic curve $\tau'\colon A\to Y'$ and a quotient invertible sheaf $\calO_{Y'}^2\to \scrL'$ whose pullback via $\tau'$ is $\calO_A^2\to \calN_s$. Then $\lambda_{\calN_s}= \tau'^\vee\circ \lambda_{Y'} \circ \tau'$ for $\lambda_{Y'}$ the polarization on $Y'$ defined by $\scrL'$. We know already that $\lambda_{\calN_s}= \tau^\vee \circ \lambda_Y \circ \tau$ for
$\lambda_Y$ a polarization on $Y$. Hence, $A\to Y'$ must coincide with the quotient $\tau\colon A \to Y$. Since $\calN_s=\tau^\ast(\scrL')$ we conclude that $\scrG_s$ is trivial, i.e., $g_s=0 $ in $E^\vee(k)$. 

We have shown that $\scrG$ is a split extension. Since $A\to Y$ is smooth and surjective, and hence faithfully flat, this implies that $\calN$ descends to $Y_S=A_S/E_S$ by faithfully flat descent (see \cite[(1.2) Proposition]{MB}).

We have shown that  $\calN$ is $\tau^\ast_S(\calL)$ for an invertible $\calO_{Y_S}$-module $\calL$. Since $\tau$ is proper with connected fibers, we have $\calL\cong \tau_{S,\ast}(\calN)$. Taking $\tau_{S,\ast}$ of the map $\calO_{A_S}^2\to \calN$ we get a morphism 
$\calO_{Y_S}^2=\tau_{S,\ast}\bigl(\calO_{A_S}^2) \arr \tau_{S,\ast}(\calN)\cong \calL$. Its pullback via $\tau_S$ is the map $\calO_{A_S}^2\to \calN$. Since $\tau_S$ is faithfully flat, we conclude that $\calO_{Y_S}^2\to \calL$ is surjective. The fact that $\calL$ has degree $d$ fiberwise over $S$ can be checked pointwise on $S$ and follows from the computation in \S \ref{sec Hilbertpol}. 

We deduce that the morphism $A_S\to \PP^1_S$ defined by  $\calO_{A_S}^2\to \calN$ factors as the composition of $\tau_S=\tau\times_k S\colon A_S\to Y_S$ and an $S$-morphism $\Psi_S\colon Y_S\to \PP^1_S$.  Since $\calL$ has degree $d$ fiberwise over $S$, we conclude that $\Psi_S$ is finite and flat of degree $d$ fiberwise over $S$. It follows from the fiberwise criterion of flatness \cite[Lemma 37.16.3]{StackProject} that $\Psi_S$ is flat. Since $\Psi_S$ is a morphism of projective schemes over $S$, it is proper. Since it is quasi-finite, it is also finite. Hence, $\Psi_S$ is finite and flat of degree $d$. Proposition \ref{prop PsiR} implies that $\Psi_S$ defines a morphism $\nu\colon S \to \WW_{Y, d}^\circ$. Since the pullback of $\calO_{Y_S}^2\to \calL$ via $\tau_S$ is  $\calO_{A_S}^2\to \calN$ we have $\rho \circ \nu={\rm Id}_S$.   Similarly, we have $\nu \circ \rho={\rm Id}_{\WW_{Y,d}^\circ}$. 

Hence, $\rho$ defines an isomorphism $S\cong \WW_{Y, d}^\circ$.  Since $\WW_{Y, d}^\circ$  is smooth and $2d$-dimensional, it follows from Proposition \ref{prop Quotosmooth} that $S'=S$. 

\end{proof}

\ 

\begin{cor} \label{cor: 098245} Let $k$ be an algebraically closed field of characteristic $p$ and $A$ an abelian surface over $k$. Let $E \subset A$ be an elliptic curve and $Y = A/E$, as before. Choose a trivialization of $\calT_A \cong \calO_A^2$. There exists a closed subset $\calZ_{Y,d}$ of $\WW^\circ_{Y,d}$ parameterizing $p$-closed foliations $\scrF$ with invariants $E$ and $d$. 
\end{cor}
\begin{proof} Given Theorem~\ref{W and quot}, this is a statement about a component $S$ of ${\rm Quot}^{P, \circ}$. If $\calO_{A_S}^2 \arr \calN^{\rm univ}$ is the universal quotient over $S$, with kernel $\calF$, then we have the $p$-curvature morphism of \S~\ref{567}, $\kappa\colon F^\ast_{abs}\calF \arr \calO_{A_S}^2/\calF$. Its vanishing locus $\calV$ is a closed subscheme of~$A_S$. Applying the semi-continuity theorem for the proper morphism $\calV \arr S$ (\cite[IV 13.1.5]{EGA}), we find that the $p$-closed locus $\calZ_{Y,d}$, characterized as the points $s\in S$ such that  
$\dim(\calV_s) = 2$, is closed. 
\end{proof}

\begin{rmk} The computations of \S~\ref{123} can be viewed as understanding the image of the projection $\calZ_{Y,d} \arr \Sym^d Y$.
\end{rmk}

\section{Quotients of abelian surfaces by foliations and $K3$ surfaces} In this part of the paper we analyze some properties of the quotient of an abelian surface $A$ by a smooth $p$-closed foliation $\scrF$ and provide an application to $K3$ surfaces. In particular, we show that if $\scrF$ is not linear then $B = A/\scrF$ is not an abelian surface, and its Albanese variety is $A^{(p)}$. We also construct $p$-closed foliations on certain Kummer $K3$ surfaces with well-controlled singularities. At least in characteristic $0$, or characteristic $p>3$, $K3$ surfaces do not carry smooth foliations, let alone smooth $p$-closed foliations, so this is the best one can hope for. 
\

\subsection{Quotients of abelian surfaces by foliations} 
Let $A$ be an abelian surface over an algebraically closed field $k$ of characteristic $p$ and let $\scrF$ be a rank~$1$ $p$-closed smooth foliation on $A$, linear or turbulent. Let $B = A/\scrF$ be the associated variety with the canonical factorization of the Frobenius morphism $A \arr A^{(p)}$, 
\[ \xymatrix{A \ar[r]^\pi & B \ar[r]^\rho & A^{(p)}.}\]

\begin{prop}\label{prop: geometry of the quotient of A by a foliation}  The variety $B$ carries a canonical $p$-closed foliation $\scrG$ such that we can identify $B/\scrG$ with $A^{(p)}$ in such a way that the canonical morphism $B\arr B/\scrG$ is $\rho$. Furthermore, as in \cite{SB}, we have the exact sequences
\begin{equation}
 \xymatrix{0 \ar[r]  & \scrF \ar[r] &\calT_A \ar[r] & N_\scrF = \pi^\ast \scrG \ar[r] & 0;}
\end{equation}
\begin{equation}\label{equation 987}
 \xymatrix{0 \ar[r] &\scrG \ar[r] &\calT_B \ar[r] & \calT_B/\scrG = \rho^\ast \scrF^{(p)} \ar[r] & 0.}
\end{equation}
The following statements hold:\begin{enumerate}
    \item If $\scrF$ is linear, $B$ is an abelian variety. 
    \item If $\scrF$ is turbulent, $B$ is not an abelian variety; in fact, $\pi^\ast K_B = -(p-1) \scrF$.
    \item $c_2(B)$ is torsion. 
\end{enumerate}
\end{prop}
\begin{proof} This is mostly well known. For the first claim one uses that $\scrF$ is the covariant Dieudonn\'e module of a group scheme $G \subseteq \Ker({\rm Fr}_A\colon A \arr A^{(p)})$ and $A/\scrF  = A/G$ is therefore an abelian variety. 

For the exact sequences, see for example \cite{Ekedahl} (especially Corollary 3.4), \cite{MP} (especially Corollary 1.11, which however has a sign mistake) or \cite{SB} (especially Proposition 9.1.2.3). Write $\scrF = \calO_A(-D)$; when $\scrF$ is linear, $D = 0$, while when $\scrF$ is turbulent $D$ is an effective fibral divisor on $A$, as we discussed above. Applying $\pi^\ast$ to (\ref{equation 987}) we get
\[ \xymatrix{0 \ar[r] &\pi^\ast\scrG  \ar[r] &\pi^\ast\calT_B \ar[r] & \pi^\ast \rho^\ast \scrF^{(p)} \ar[r] & 0.}\]
As $\rho\circ \pi = {\rm Fr}_A$, we find
\[ \xymatrix{0 \ar[r] &\scrF^\vee  \ar[r] &\pi^\ast\calT_B \ar[r] & p\scrF \ar[r] & 0.}\]
Taking determinants we find the previous formula $\pi^\ast K_B = -(p-1) \scrF$, but we also find
\[c_2(\pi^\ast \calT_B) = \pi^\ast(c_2(B)) = \scrF^\vee . p\scrF = -pD^2 = 0.\] 
Use that $\pi$ is a finite flat morphism of degree $p$ and so $\pi_\ast \pi^\ast$ acts as multiplication by $p$ on ${\rm CH}_0(B)$ to conclude that $c_2(B)$ is $p$-torsion. 
\end{proof}

\vspace{3cm}

\begin{prop} \label{Prop: albanese} Assume that $\scrF$ is a turbulent foliation. Taking $0_B$ to be the image of $0_A$ under the morphism $\pi\colon A \arr B$, we have
\[ {\rm Alb}(B, 0_B) \cong A^{(p)}.\]
\end{prop}
\begin{proof} We have a commutative diagram 
\[ \xymatrix@C=2cm{A \ar[r]^\cong\ar[d]_\pi^{/\scrF} & {\rm Alb}(A) = A\ar[d]^\varphi\\
B \ar[r]^\alpha \ar[d]_\rho^{/\scrG} & {\rm Alb}(B, 0_B)\ar[d]^\gamma \\
A^{(p)} \ar[r]^\cong & {\rm Alb}(A^{(p)}) = A^{(p)}.
}\]
This diagram shows that $\alpha$ is a finite flat morphism. Note that the morphisms $\rho\circ \pi$ and $\gamma\circ \varphi$ are simply  ${\rm Fr}_A$. Therefore, $p^2 = \deg(\gamma) \deg(\alpha) \deg(\pi) = p\deg(\gamma) \deg(\alpha)$. Thus, as $\alpha$ is \textit{not} an isomorphism, $\deg(\alpha) =~p$ and $\deg(\gamma) = 1$, and we conclude that $ {\rm Alb}(B, 0_B) \cong A^{(p)}$.
\end{proof}

\ 

\subsection{Foliations on Kummer $K3$ surfaces} 
We shall use our results to construct Kummer $K3$ surfaces with $p$-closed foliations with rather mild singularities. We begin with a quick recall concerning $K3$ surfaces and motivate our construction by explaining why the introduction of singularities is necessary. Before getting to the case of $K3$ surfaces proper, we discuss an example in the plane as it will inform our calculations later. 

\subsubsection{A quick recall about $K3$ surfaces.} A good reference is \cite{H}. Recall that a $K3$ surface over an algebraically closed field~$k$ is a projective non-singular connected surface $X$ such that its canonical bundle $K_X = \bigwedge^2 \Omega_X$ is trivial and its irregularity is zero. That is,
\[ K_X \cong \calO_X, \qquad H^1(X, \calO_X) = 0.\]
One can reformulate the first property as saying that $X$ has a nowhere-vanishing regular $2$-form.

Using Serre duality and the triviality of $K_X$, one finds $h^0(X, \calO_X) = h^2(X, \calO_X) = 1$ and thus the arithmetic genus
\[ g_a(X) = \chi(\calO_X) -1 = 1.\]
Using Noether's formula $2 = g_a(X) +1= \chi(\calO_X) =  \frac{1}{12}(c_1(X)^2 + c_2(X))$ one concludes that
\[ c_1(X) = 0, \quad c_2(X) = 24.\]

\medskip

If $k = \CC$, $c_2(X)$ is the topological Euler characteristic. As $b_0(X) = b_4(X) = 1$ (as $X$ is connected) and $b_1(X) = b_3(X) = 0$ (the first is deduced by taking the cohomology of the exponential sequence $0 \arr \ZZ \arr \calO_X \arr \calO_X^\times \arr 0$, and the second follows by Poincar\'e duality), one can thus deduce that $b_2(X) = 22$. Further, as $H^1(X, \calO_X) = 0$, ${\rm Pic}(X)$ embeds as a saturated sublattice in $H^2(X, \ZZ)$, which is a torsion-free group, known to be isomorphic to the even unimodular lattice $E_8(-1)^{\oplus 2} \oplus U^3$, where $U$ is a hyperbolic plane and $E_8(-1)$ is the $E_8$-lattice with minus its quadratic form. It is thus of signature $(3,19)$. By the Hodge index theorem, ${\rm Pic}(X)$ has signature $(1, \rho-1)$ for some $\rho$ and it is known that $\rho$ can take any value between $1$ and $20$, with $\rho=1$ being the generic case. 

In characteristic $p$ the bound on $\rho$ fails and in fact $\rho$ could be as large as $22$. Examples are the supersingular Kummer surfaces, obtained by the procedure we study below where the original abelian surface $A$ is supersingular (that is, in our notation, both $E$ and $Y$ are supersingular elliptic curves).

\subsubsection{Examples of $K3$ surfaces.} The main examples of $K3$ surfaces are the following.
\begin{itemize}
\item Kummer $K3$ surfaces. This is our main example. Let $k$ be an algebraically closed field of characteristic different from $2$. One starts with an abelian surface $A$ over $k$ and takes the quotient $Q = A/\{\pm 1\}$ that has 16 singular points. The blow-up $\tilde Q$ of $Q$ at those singular points is a so-called Kummer $K3$ surface.  
\item The double cover $X$ of the projective plane $\PP^2$ branched along a smooth sextic curve, in characteristic different than $2$.
\item A smooth quartic surface in $\PP^3$.
\item Let $k$ be an algebraically closed field of characteristic zero or $p>3$. Let $\pi\colon X \arr \PP^1_k$
 be a relatively
minimal elliptic fibration with a section, where $X$ is a smooth projective surface. If $\chi(X, \calO_X) = 2$, then $X$ is a $K3$ surface. Its singular fibers have Kodaira types and their Euler characteristics sum to $24$. Here Euler characteristic is topological over $\CC$ and is interpreted using \'etale cohomology in general. A general elliptic $K3$ surface with a section has 24 singular fibers, all of type $I_1$, that is, irreducible rational curves with one node.
\end{itemize}

\subsubsection{Non-existence of smooth foliations on $K3$ surfaces.} In characteristic $0$, or $p> 3$, on a  $K3$ surface $X$ there are no smooth $p$-foliations. In characteristic $0$, where the notion of $p$-foliation does not make sense, we are talking simply about smooth foliations and the statement is that none exist for any $K3$ surface.  The proof is based on analytic techniques. 

Let $\scrF$ be a foliation on a complex projective non-singular surface $X$ with isolated singularities. Let $\calT_\scrF$ be its tangent bundle (we introduce the notation to conform with cited literature; for us there is no difference between $\scrF$ and $\calT_\scrF$), which is a line bundle. On the other hand, if we define $N_\scrF$ to be $\calT_X/\calT_\scrF$ then $N_\scrF$, which we call here the virtual normal bundle, is not a line bundle, but only a coherent sheaf that is free outside the singular locus ${\rm Sing}(\scrF)$ of $\scrF$; cf. \S~\ref{sec:foliations generalities}. In $K(X)$ (the $K$-group of coherent sheaves on $X$) we can write
\[ [N_\scrF]  = [\calT_X]-[\calT_\scrF], \]
and the formalism of Chern classes applies. In particular, the total Chern class $c(N_\scrF)$ is defined by 
\[ c(N_\scrF) = \frac{c(\calT_X)}{c(\calT_\scrF)} = \frac{1 + c_1(X) + c_2(X)}{1 + c_1(\calT_\scrF)},\]
yielding the classes
\[ c_1(N_\scrF) = c_1(X) - c_1(\calT_\scrF), \qquad c_2(N_{\scrF}) = c_2(X) - c_1(\calT_\scrF) \cdot c_1(N_{\scrF}). \]
The Baum--Bott Theorem states the following formulas.
\[c_1(N_{\scrF})^2=
\sum_{p \in {\rm Sing}(\scrF)} {\rm Res}_{\scrF}(c_1(N_{\scrF})^2, p), 
\]

\[c_2(N_{\scrF})=
\sum_{p \in {\rm Sing}(\scrF)} {\rm Res}_{\scrF}(c_2(N_{\scrF}), p) .
\]
The exact definition of the residues is of no importance to us, as we now specialize to the case where $\scrF$ is a smooth foliation in which case the sum is empty and $N_\scrF$ is a line bundle and so $c_2(N_\scrF) = 0$. We then find
\[ c_1(N_\scrF)^2 = 0, \qquad c_2(N_\scrF) = 0.\]
Assuming further that $X$ is a $K3$ surface, we have $c_1(X) = 0, c_2(X) = 24$. Thus, we find that 
$ c_1(N_\scrF) = c_1(X) - c_1(T_\scrF) = - c_1(T_\scrF)$ and substituting that in the formula for $c_2(N_\scrF)$ we get that
\[ 0 = c_2(N_\scrF) = c_2(X) - c_1(T_\scrF) \cdot c_1(N_{\scrF}) = 24 + c_1(N_\scrF)^2.\]
But $c_1(N_\scrF)^2 = 0$. Contradiction. 

\medskip

\id In characteristic $p$ the argument can be made to work with certain modifications, namely, when one views Chern classes as taking values in $ H^r(X,\Omega_X^r)$. This interpretation uses a cycle class map $K_0({\rm Vect}(X)) \arr H^{2\ast}_{\text{dR}}(X)$ \cite[50.21]{StackProject}. The Bott vanishing result $c_1(N_\scrF)^2 = 0$ follows from \cite[Theorem 4.1]{FMST}. One again derives the equality $24 = 0$, which now takes place in $k$. Thus, if $p = {\rm char}(k) > 3$ we conclude that there are no smooth foliations on $K3$ surfaces in characteristic~$p$ ($p$-closed or not!). We remark that the non-existence of $p$-closed foliations for $p>3$ also follows from \cite{RS}.

\subsubsection{Foliations on the plane and the quadratic cone} Let $k$ be a field of characteristic different from~$2$. Consider a foliation $\scrF_{\AA^2}$ in the plane $\AA^2_{x, y}$, with coordinates $x, y$, given by the differential form
\[ \omega = f(x, y) dx + dy,\]
where $f(x, y)$ is an even polynomial: $f(-x, -y) = f(x, y)$ and is thus a polynomial in $x^2, xy$ and $y^2$.
Let $G = \{ \pm 1\}$ act on $\AA^2$ by 
\[ [-1](x, y) = (-x, -y).\]
Note that $[-1]^\ast \omega = -\omega$ and consequently, $\scrF_{\AA^2}\cong [-1]^\ast \scrF_{\AA^2}$.

Let $Q = \AA^2/G$ be the quotient of the plane by the action of $G$. Then $Q$ is the quadratic cone with coordinate ring 
\[ k[x^2,xy,y^2]\cong k[u,v,w]/(uw-v^2).\] 
The morphism $\varphi\colon \AA^2 \arr Q\subset \AA^3_{u, v, w}$ is given by
\[ \varphi(x, y) = (x^2, xy, y^2) = (u, v, w),\]
and is a Galois cover that is \'etale outside the point $0$. The function $f$ can be expressed by a polynomial $F$ in $u, v, w$. The differentials $du, dv, dw$ pull back to differentials on $\AA^2$ given by
\[ du = dx^2 = 2x \,dx, \quad dv = d(xy) = x \,dy + y \,dx, \quad dw = dy^2 = 2y\, dy.\] 
We may replace $\omega$ by $2x\,\omega$, which is now invariant under the action of $G$, and 
\[ 2x\,\omega = f\cdot 2x\,dx + 2x\,dy = F\,du + 2\,dv-2y\,dx.\]
As $y/x=v/u$ we have $2y\, dx = (v/u) 2x\,dx$ and so $2x\,\omega = F\,du + 2\,dv-(v/u) \,du$. Multiplying by $u$, we get that 
\[ \eta:= 2ux\,\omega = uF\,du + 2u\,dv-v \,du = (uF-v)\,du + 2u\,dv.\]
It is a regular $1$-form on $Q$ that vanishes along the line on the cone given by $u=v=0$, and only there. Similarly, consider $\eta^\prime$ given by
\[ \eta^\prime = 2wy \omega = (w - F v)dw  + 2F w dv,\]
which is also a regular $1$-form on $Q$ that vanishes along the line $v = w = 0$, and only there. As $\eta$ and $\eta^\prime$ differ by a meromorphic function, they define together a foliation $\scrF_Q$ on $Q$, with an isolated singular locus equal to the vertex of the cone $u=v=w=0$. This foliation is the descent of $\scrF_{\AA^2}$.

Let $\tilde Q$ be the blow-up of $Q$ at its vertex. We can also blow up $\AA^2$ at the origin to get a variety~$\tilde \AA^2$. The action of $G$ extends to $\tilde \AA^2$ and, as the points on the exceptional divisor $E$ correspond to the slopes $x/y$, we find that the action of $G$ is trivial on $E$. We have
\[ \Phi\colon \tilde \AA^2 \arr \tilde Q = \tilde \AA^2/G,\]
where the map is completely ramified along $E$; namely, if we call $D$ the exceptional divisor on $\tilde Q$, then $\Phi^{-1}(D) = 2E$.

As $\tilde Q$ is birational to $Q$, the foliation $\scrF_Q$ induces a foliation $\scrF_{\tilde Q}$ on $\tilde Q$ that is smooth at the very least outside~$D$. In the same way, the foliation $\scrF_{\AA^2}$ induces a foliation $\scrF_{\tilde \AA^2}$. All this is summarized in the commutative diagram of foliated varieties
\[ \xymatrix@C=3cm{(\AA^2, \scrF_{\AA^2}) \ar[d]^{/G} & (\tilde \AA^2, \scrF_{\tilde \AA^2}) \ar[l]^{\text{blow-down}}\ar[d]^{/G}\\ (Q, \scrF_Q) & (\tilde Q, \scrF_{\tilde Q})\ar[l]^{\text{blow-down}}. }
\]
It is easier to calculate the singularities of $\scrF_{\tilde \AA^2}$.

\subsubsection{Singularities of $\mathscr{F}_{\tilde{\mathbb{A}}^2}$} The blow-up $\tilde \AA^2$ of $\AA^2$ is defined in $\AA^2 \times \PP^1$, with the coordinates $(x, y; s:t)$, by the equation $xt = ys$. One chart is given by $t\neq 0$, so without loss of generality $t=1$ and we obtain the chart in $\AA^3$ with coordinates $x, y, s$ described by $x = ys$. This chart is isomorphic to $\AA^2$ by the map $(x, y, s) \mapsto (y, s)$, and $y=0$ is the part of the exceptional divisor in this chart. We can express the form $\omega$ by
\[\omega = f\,dx + dy = f(ys, y) (y\,ds + s\,dy) + dy = (sf(ys,y) + 1)dy + yf(ys,y)\, ds.\] 
This is a regular differential that vanishes only where $y=sf(ys,y)+1 = 0$; namely, it could only vanish at the point $sf(0,0)+1 = 0$ on the exceptional divisor. If $f(0,0) = 0$ this is not possible. Otherwise, we get a unique point where the differential vanishes and it is the point \[(x, y; s:1) = (0, 0; -1/f(0,0):1).\]

The other chart is obtained by letting $s = 1$ and so $xt = y$ and this chart is isomorphic to $\AA^2$ by the map $(x, y, t) \mapsto (x, t)$. We can write now
\[\omega  = f\,dx + dy = f(x, xt)\,dx + t\,dx + x\,dt = (f(x, xt) + t)dx + x \,dt,\] which is a regular differential that can only vanish when $x = 0$ and $f(0,0) = -t$. We get a unique vanishing point 
\[ (x, y; 1:t) = (0, 0; 1: - f(0,0)),\] which, if $f(0,0) \neq 0$, is also the singular point found in the other chart. 

Either way, \textit{the foliation has a unique singular point on $\tilde \AA^2$, which is a point on the special fiber. The same holds for $\tilde Q$.} 

\subsubsection{Integral curves}\label{sec integralcurves} Note that a derivation $\partial$ annihilated by $\omega$ is given by \[\partial = \partial_x - f\partial_y;\] 
it is a tangent generator of the foliation. We look for polynomials $g\not\in k[x^p,y^p]$ such that $\partial g=0$ in order to define integral curves for the foliation.

\begin{lem}\label{lem: solutions to partial} There exists $g\not\in k[x^p,y^p]$ such that $\partial g=0$  if and only if  $\partial^p$ is the zero derivation. Every such $g$ satisfies $\partial_y g \neq 0$. Moreover, if $f$ is even, there exists such $g$ which is odd. 

Furthermore, if $\partial^p=0$ then $k[x,y]\cdot \partial$ is a smooth $p$-closed foliation on $\AA^2$, and thus the set $k[x,y]^\partial$ of all solutions to $\partial g = 0$ is a ring satisfying
\[ k[x^p, y^p] \subset k[x, y]^\partial \subset k[x, y],\]
and is a locally-free $k[x^p, y^p]$-module of rank $p$.
\end{lem}
\begin{proof} Since $\partial(x)=1$, we have $\partial^p(x)=0$. Hence, given a polynomial $F$, we get $\partial^p F=\partial^p(y) \partial_y F$. Hence, $\partial^p=0$ if and only if $\partial^p(y)=0$. 

If $g\not\in k[x^p,y^p]$ is a solution then $0=\partial^p g= \partial^p(y) \partial_y g$. If $\partial^p(y)$ were not $0$ then $\partial_y g=0$, and it follows that $0=\partial g=\partial_ x g=0$. This implies $g\in k[x^p,y^p]$ (contradiction).  The same argument implies that if $g\not\in k[x^p,y^p]$ is a solution then $\partial_y g\neq 0$ as claimed. 

Vice versa, if   $\partial^p(y)=0$, one checks that 
\begin{equation}\label{eqn: explicit g}
g=\sum_{i=0}^{p-1}\frac{(-x)^i}{i!}\,\partial^i(y)
\end{equation} 
is a solution to $\partial g=0$ and $g\not\in k[x^p,y^p]$ as the term for $i=0$ is $y$. 

One proves by induction that $\partial^i(y)$ is odd for $i$ even, and even for $i$ odd. It thus follows that $g$ is odd. 

As to the second part of the lemma, a formula of Hochschild \cite[Theorem 25.5]{Matsumura} gives
\[(h\partial)^p=h^p\partial^p+(h\partial)^{p-1}(h)\,\partial.
\]
from which it follows that $k[x,y]\cdot \partial$ defines a $p$-closed foliation, which is smooth as $\omega$ is nowhere vanishing; denote by $k[x,y]^\partial$ its ring of constants. The rest follows from the connection between foliations and purely inseparable maps $\AA^2\arr B$, factoring Frobenius.
\end{proof} 

\begin{prop} Let $f(x, y) \in k[x, y]$ be such that the derivation $\partial = \partial_x - f\partial_y$ satisfies $\partial^p = 0$. Take a solution $g$ as in Equation (\ref{eqn: explicit g}). The open set 
\[ U = \{ (x, y) \colon \partial_y g \neq 0\}\]
contains $(0,0)$. The curves,
\[C_\kappa\colon \quad g = \kappa,\]
are integral curves on $U$ for the foliation defined by $\omega = fdx + dy$. The involution $[-1]$ takes $C_\kappa$ to $C_{-\kappa}$. The integral curve $C_0$ passes through $(0, 0)$ and its strict transform meet the exceptional fiber at the point $(0, 0; s:t) = (0, 0; 1:-f(0))$, which is precisely the singular point of the foliation on $\tilde \AA^2$. 
\end{prop}
\begin{proof} The relation $0 = \partial g = \partial_x g - f\partial_y g$ gives
\[ f\partial_y g= \partial_xg.\]
It then follows that 
\[ dg = \partial_y g (f dx+ dy),\]
which is nowhere vanishing along the open set $U$, and only there. As $g = y + h.o.t.$, where the higher order terms belong to the ideal $(x, y^2)$, we find that $\partial_yg (0, 0) = 1$ and thus the point $(0,0) \in U$. 
Thus, the level curves $C_\kappa$ are smooth integrals curves for $\partial$ along $U$. The assertion regarding $C_0$ is a straight-forward calculation. 
\end{proof}
\begin{rmk} If $f(x, y) \in k[x]$, then the condition $\partial^p = 0$ is equivalent to $\partial^p(y) = 0$. However, $\partial^p(y) = - f^{(p-1)}$ and the condition $\partial^p = 0$ is equivalent to the condition that every power $x^a$ appearing in $f$ satisfies $a \not\equiv p-1 \pmod{p}$. This condition is equivalent to $f$ having an anti-derivative $h$ relative to $\partial_x$. In fact, in this case a solution to $\partial(g) = 0$ can be chosen to be $g = h+y$.
\end{rmk}

\medskip

\id Passing now to $\tilde Q$, or rather to $\tilde \AA^2$, let us consider the exceptional divisor in the first chart. In this chart, isomorphic to $\AA^2$ with coordinates $y, s$, the foliation is defined by the form \[\omega = (sf(ys,y) + 1)\,dy + yf(ys,y)\, ds.\] 
A derivation $\partial$ killed by $\omega$ is given by
\[ \partial = yf(ys,y) \partial_y- (sf(ys,y) + 1) \partial_s.\]
The exceptional divisor is given in this chart by $y=0$ and we have $\partial(y) = yf(ys,y)$, which vanishes along the exceptional curve. Thus, \textit{the exceptional divisor is an integral curve.} Note this last conclusion does not depend on the existence of the integral curves $C_\kappa$; namely, it holds true for any even $f$, whether $\partial^p = 0$ or not.

\

\subsubsection{Construction of singular $p$-foliations on Kummer surfaces} As above, assume that $p>2$. As we have done before, let $A$ be an abelian surface over an algebraically closed field $k$ of characteristic $p$, such that there is an exact sequence 
\[ \xymatrix{0 \ar[r]& E \ar[r] & A \ar[r]^\tau & Y \ar[r] & 0},\]
where we treat the first arrow as an inclusion; $E, Y$ are elliptic curves. Let $f$ be an even function on $Y$, relative to its involution $[-1]$. Otherwise phrased, if $Y$ is given by a Weierstrass equation $y^2 = x^3 + ax^2 + bx + c$ then $f$ is a function of~$x$. With our previous conventions, we have the differential 
\[\eta = \tau^\ast (f(x) \omega_Y) + \tilde \omega_E.  \]
It is a meromorphic differential on $A$ that under suitable conditions (see Theorems~\ref{p closed thm 1},~\ref{p closed thm 2}) defines a $p$-closed foliation; it has the property 
\[ [-1]^\ast \eta = -\eta,\]
and so the foliation $\scrF$ it defines is preserved under $[-1]$. It descends to a foliation $\scrF_Q$ on the singular Kummer surface 
\[ Q = A/\{ \pm 1\}.\]
The morphism 
\[ \pi\colon A \arr Q\]
is $2:1$ and is ramified at the $16$ points $A[2]$ that become singular points on $Q$ (elsewhere it is an \'etale double cover). Thus, $Q$ is a surface with $16$ singular points comprising $Q^{sing}$. The foliation~$\scrF$ defines a foliation $\scrF_{\tilde Q}$ on the non-singular Kummer $K3$ surface $\tilde Q$ obtained by blowing up the $16$ singular points of $Q$. A priori, the foliation $\scrF_{\tilde Q}$ is only defined on the complement of the exceptional divisor $E$ on $\tilde Q$ simply because $\tilde Q - E \cong Q - Q^{sing}$. However, it extends uniquely, by a general procedure, to a foliation $\scrF_{\tilde Q}$ on $\tilde Q$ with isolated singularities. 

\begin{lem}\label{lem:local coord} Let $P$ be a $2$-torsion point of $A$ and assume $f$ is regular at $\tau(P)$. There is a choice of a local coordinate $x_1$ on $Y$ at $\tau(P)$ such that $[-1]^\ast(x_1) = -x_1$. Denoting by $x_1$ also the pullback $\tau^\ast(x_1)$, there is a choice of a local coordinate $x_2$ at $P$ with the properties (i) $[-1]^\ast(x_2) = -x_2$ and (ii)  for suitable power series $\alpha, \beta \in k[\![x_1, x_2]\!]$, with $\beta$ invertible,
\[ \eta = \alpha(x_1, x_2) dx_1 + \beta(x_1, x_2) dx_2. \]

\end{lem}
\begin{proof} Denote  $\sigma = [-1]$. Note that $\sigma$ acts by $-1$ on the tangent space of $A$ (or of $Y$) at any $2$-torsion point. Taking any local uniformizer $s$ on $Y$ at $\tau(P)$, we may replace it by $(s-\sigma(s))/2$, because modulo the square of the maximal ideal, $(s-\sigma(s))/2\equiv s$ and so, by Nakayama's lemma, $(s-\sigma(s))/2$ is also a uniformizer. In this notation $\tau^\ast\omega_Y = \gamma_1(x_1) dx_1$.

Choose a local parameter $x_2$ at $P$ on $A$ such that $\{x_1, x_2\}$ are local coordinates at $P$ and $dx_2 \equiv \tilde \omega_E$ in the cotangent space at $P$. By the same argument of replacing $x_2$ by $(x_2 - \sigma(x_2))/2$, we may achieve $\sigma(x_2) = -x_2$ and then $\tilde \omega_E = \gamma_2(x_1, x_2) dx_1+\beta(x_1, x_2) dx_2$, where $\beta(x_1, x_2) \in k[\![x_1,x_2]\!]^\times$ and $\gamma_2 \in (x_1, x_2) \subset k[\![x_1, x_2]\!]$.
In this notation \[\eta = \tau^\ast(f(x_1) \omega_Y) + \tilde \omega_E = (f(x_1) \gamma_1(x_1) + \gamma_2(x_1, x_2))\;dx_1 + \beta(x_1, x_2)\; dx_2.\]
\end{proof}

Lemma~\ref{lem:local coord} provides us with local coordinates at any ramification point $P$ of the morphism $A \arr Q$ that show that at the complete local ring of $P$ the description of the foliation is essentially as given in the plane example studied above, once we replace $\eta$ with 
$\omega = \beta(x_1,x_2)^{-1}\eta = f(x_1, x_2)dx_1 + dx_2$,
with the minor difference that $f(x_1, x_2)$ is now a power series in $x_1, x_2$ and not just a polynomial. None of this affects the discussion. We can therefore conclude the following proposition, once we note that being $p$-closed can be checked on a dense open set.  

\begin{prop}  \label{prop: foliations on kummer}With the notation above, assume that $f$ is an even function that is regular 
at every point of $Y[2]$. Then $\scrF_{\tilde Q}$ is a singular foliation on the $K3$ surface $\tilde Q$ obtained from $A$ and $\scrF_{\tilde Q}$ has exactly $16$ singular points, one on each exceptional fiber of the morphism $\tilde Q \arr Q$. Each exceptional curve is a rational integral curve of $\scrF_{\tilde Q}$. Furthermore, if the foliation $\scrF$ on $A$ is $p$-closed, so is $\scrF_{\tilde Q}$.
\end{prop}


\begin{thebibliography}{99}

\bibitem{Bost} Bost, Jean-Beno\^{i}t:
Algebraic leaves of algebraic foliations over number fields.
Publ. Math. Inst. Hautes \'Etudes Sci. No. 93 (2001), 161--221.

\bibitem{Bru1} Brunella, Marco:
Minimal models of foliated algebraic surfaces.
Bull. Soc. Math. France 127 (1999), no. 2, 289--305.

\bibitem{Brunella} Brunella, Marco:
Codimension one foliations on complex tori.
Ann. Fac. Sci. Toulouse Math. (6) 19 (2010), no. 2, 405--418.

\bibitem{Cartier} Cartier, Pierre:
Une nouvelle opération sur les formes diff\'erentielles.
C. R. Acad. Sci. Paris 244 (1957), 426--428.

\bibitem{EGA} Grothendieck, A.; Dieudonn\'e, J.: \'Eléments de g\'eom\'etrie alg\'ebrique: IV. Troisi\`eme partie. \'Etude locale des sch\'emas et des morphismes de sch\'emas. Publ. Math. Inst. Hautes \'Etudes Sci. No. 28 (1966), 5--255.

\bibitem{EvdGM} Edixhoven, B.; van der Geer, G.; and Moonen, B.: \textit{Abelian Varieties} (draft of a book). \url{http://van-der-geer.nl/~gerard/AV.pdf}.

\bibitem{Ekedahl} Ekedahl, Torsten:
Foliations and inseparable morphisms. Algebraic geometry, Bowdoin, 1985 (Brunswick, Maine, 1985), 139--149.
Proc. Sympos. Pure Math., 46, Part 2.
American Mathematical Society, Providence, RI, 1987.

\bibitem{FMST}  Fassarella, Thiago;  Mendson, Wodson;  dos Santos, Jo\~ao Pedro;  Touzet, Fr\'edéric: Foliations with small singular set in arbitrary characteristic. 2026. \url{https://arxiv.org/abs/2607.18571}

\bibitem{GdS} Goren, Eyal Z.; de Shalit, Ehud:
Foliations on Shimura varieties in positive characteristic.
J. Algebraic Geom. 33 (2024), no. 3, 401--454.

\bibitem{HartshorneReflexive} Hartshorne, Robin:
Stable reflexive sheaves.
Math. Ann. 254 (1980), no. 2, 121--176.

\bibitem{HindrySilverman} Hindry, Marc; Silverman, Joseph H.:
\textit{Diophantine geometry. An introduction.}
Grad. Texts in Math., 201.
Springer-Verlag, New York, 2000. 

\bibitem{H} Huybrechts, Daniel: \textit{Lectures on K3 Surfaces}, Cambridge Studies in Advanced Mathematics 158, Cambridge University Press, 2016.

\bibitem{HL} Huybrechts, Daniel; Lehn, Manfred:
\textit{The geometry of moduli spaces of sheaves.}
Second edition.
Cambridge Math. Lib.
Cambridge University Press, Cambridge, 2010.

\bibitem{Katz}Katz, Nicholas M.:
Nilpotent connections and the monodromy theorem: Applications of a result of Turrittin.
Inst. Hautes \'Etudes Sci. Publ. Math. No. 39 (1970), 175--232.

\bibitem{Katz2} Katz, Nicholas M.:
Algebraic solutions of differential equations ($p$-curvature and the Hodge filtration).
Invent. Math. 18 (1972), 1--118.

\bibitem{Joelle} Lim, Joelle: McGill Ph.D. thesis. \textit{In preparation}.

\bibitem{LBP} Lo Bianco, Federico; Pereira, Jorge Vit\'orio:
Smooth foliations on homogeneous compact K\"ahler manifolds.
Ann. Fac. Sci. Toulouse Math. (6) 25 (2016), no. 1, 141--159.

\bibitem{Manin} Manin, Ju. I.: The Hasse–Witt matrix of an algebraic curve. 
Izv. Akad. Nauk SSSR Ser. Mat. 25 (1961), 153--172. Amer. Math. Soc. Transl. (2),
45 (1965), 245--264. Translated by J. W. S. Cassels.

\bibitem{Matsumura} Matsumura, Hideyuki:
\textit{Commutative ring theory}.
Cambridge Stud. Adv. Math., 8.
Cambridge University Press, Cambridge, 1989.

\bibitem{McQ} McQuillan, Michael:
Canonical models of foliations.
Pure Appl. Math. Q. 4 (2008), no. 3, Special Issue: In honor of Fedor Bogomolov. Part 2, 877--1012.

\bibitem{MeP} Mendson, Wodson; Pereira, Jorge Vit\'orio:
Codimension one foliations in positive characteristic.
J. Inst. Math. Jussieu 23 (2024), no. 4, 1705--1750.

\bibitem{Milne} Milne, James S.: Jacobian Varieties in ``Arithmetic geometry'' (Storrs, Conn., 1984), Springer, New York, 1986, 167--212.

\bibitem{Miyaoka} Miyaoka, Yoichi:
Deformations of a morphism along a foliation and applications. Algebraic geometry, Bowdoin, 1985 (Brunswick, Maine, 1985), 245--268.
Proc. Sympos. Pure Math., 46, Part 1.
American Mathematical Society, Providence, RI, 1987.

\bibitem{MP} Miyaoka, Yoichi; Peternell, Thomas:
\textit{Geometry of higher-dimensional algebraic varieties.}
DMV Sem., 26.
Birkh\"auser Verlag, Basel, 1997.

\bibitem{MB} Moret-Bailly, Laurent:
Pinceaux de vari\'et\'es ab\'eliennes.
Ast\'erisque No. 129 (1985).

\bibitem{MumfordCurvesOnSurfaces} Mumford, David:
\textit{Lectures on curves on an algebraic surface.}
With a section by G. M. Bergman.
Ann. of Math. Stud., No. 59.
Princeton University Press, Princeton, NJ, 1966. 

\bibitem{MumfordAbelianVarieties} Mumford, David:
\textit{Abelian varieties.}
With appendices by C. P. Ramanujam and Yuri Manin. Corrected reprint of the second (1974) edition.
Tata Inst. Fund. Res. Stud. Math., 5.
Published for the Tata Institute of Fundamental Research, Bombay; by Hindustan Book Agency, New Delhi, 2008.

\bibitem{OortCGS} Oort, F.:
\textit{Commutative group schemes}.
Lecture Notes in Math., 15.
Springer-Verlag, Berlin-New York, 1966.

\bibitem{EO} Oort, Frans:
A stratification of a moduli space of polarized abelian varieties in positive characteristic. \textit{Moduli of curves and abelian varieties}, 47--64.
Aspects Math., E33.
Friedr. Vieweg \& Sohn, Braunschweig, 1999.

\bibitem{RS} Rudakov, A. N.; Shafarevich, I. R.:
Inseparable morphisms of algebraic surfaces.
Izv. Akad. Nauk SSSR Ser. Mat. 40 (1976), no. 6, 1269--1307, 1439.

\bibitem{Seshadri} Seshadri, Conjeerveram Srirangachari: L'op\'eration de Cartier. Applications. \textit{In}
S\'eminaire Claude Chevalley, Vari\'et\'es de Picard, Volume 4 (1958-1959), Talk no. 6, 26 p.

\bibitem{SB} Shepherd-Barron, N. I.: Miyaoka's theorems on the generic seminegativity of $T_X$ and on the Kodaira dimension of minimal regular threefolds. In \textit{Flips and abundance for algebraic threefolds.}
Papers from the Second Summer Seminar on Algebraic Geometry held at the University of Utah, Salt Lake City, Utah, August 1991, Ast\'erisque No. 211 (1992), 103--114.

\bibitem{StackProject} The Stacks Project. \url{https://stacks.math.columbia.edu}.
\end{thebibliography}
\end{document}